\documentclass[12pt]{amsart}
\usepackage{amsfonts, amssymb, latexsym, graphicx, hyperref, amsmath}
\usepackage{mathtools}
\usepackage[vcentermath]{youngtab}
\usepackage{enumitem}
\usepackage{subcaption}
\usepackage[utf8]{inputenc}
\usepackage[export]{adjustbox}
\usepackage{tikz}
\usetikzlibrary{decorations.markings}
\usepackage{tikz-cd}
\usetikzlibrary{calc, shapes, backgrounds,arrows,positioning,decorations.pathmorphing,patterns}

\newtheorem{theorem}{Theorem}[section]
\newtheorem{proposition}[theorem]{Proposition}

\newtheorem{lemma}[theorem]{Lemma}

\newtheorem*{claim*}{Claim}
\newtheorem{corollary}[theorem]{Corollary}
\newtheorem{Main Conjecture}[theorem]{Main Conjecture}
\newtheorem{conjecture}[theorem]{Conjecture}
\newtheorem{deflemma}[theorem]{Definition-Lemma}

\theoremstyle{remark}
\newtheorem{definition}[theorem]{Definition}

\newtheorem{example}[theorem]{Example}

\newtheorem{remark}[theorem]{Remark}
\theoremstyle{plain}
\newtheorem{question}{Question}

\newcommand\BKdeg{{\operatorname{BK-deg}}}

\newcommand\Span{{\operatorname{Span}}}
\newcommand\codim{{\operatorname{codim}}}
\newcommand\mult{{\operatorname{mult}}}

\newcommand\Hom{{\operatorname {Hom}}}

\newcommand\Gr{{\operatorname {Gr}}}
\newcommand\SL{{\operatorname {SL}}}
\newcommand\Sp{{\operatorname {Sp}}}
\newcommand\GL{{\operatorname {GL}}}
\newcommand\Lie{{\operatorname {Lie}}}

\renewcommand\sp{{\mathfrak{sp}}}
\renewcommand{\lg}{\mathfrak{g}}

\newcommand\diag{{\operatorname{diag}}}
\newcommand\EH{{\operatorname{EH}}}
\newcommand\Pic{{\operatorname{Pic}}}
\newcommand\Ho{{\operatorname{H}}}
\newcommand\nls{{\operatorname{NL-semigroup}}}
\newcommand\lrs{{\operatorname{LR-semigroup}}}
\newcommand\sps{{\sp\operatorname{-semigroup}}}
\newcommand\lgs{{\lg\operatorname{-semigroup}}}
\newcommand\nlss{{\operatorname{NL-sat}}}
\newcommand\lrss{{\operatorname{LR-sat}}}
\newcommand\spss{{\sp\operatorname{-sat}}}
\newcommand\lgss{{\lg\operatorname{-sat}}}
\newcommand\gits{{\operatorname{GIT-semigroup}}}
\newcommand\gitss{{\operatorname{GIT-sat}}}
\newcommand\Par{{\operatorname{Par}}}
\newcommand\Supp{{\operatorname{Supp}}}

\newcommand\inv{{^{-1}}}

\newcommand\Li{{\mathcal L}}

\newcommand\Schub{\operatorname{Schub}}
\newcommand\lu{{\mathfrak u}}\newcommand\lt{{\mathfrak t}}
\newcommand\lb{{\mathfrak b}}

\newcommand\ZZ{{\mathbb Z}}\newcommand\RR{{\mathbb R}}
\newcommand\CC{\mathbb C}\newcommand\NN{\mathbb N}
\newcommand\QQ{\mathbb Q}
\newcommand\longto{\longrightarrow}

\newcommand\Part{{\mathcal
    S}}

\newcommand{\cellsize}{11}
\newlength{\cellsz} 
\newsavebox{\cell}
\sbox{\cell}{\begin{picture}(\cellsize,\cellsize)
\put(0,0){\line(1,0){\cellsize}}
\put(0,0){\line(0,1){\cellsize}}
\put(\cellsize,0){\line(0,1){\cellsize}}
\put(0,\cellsize){\line(1,0){\cellsize}}
\end{picture}}
\newcommand\cellify[1]{\def\thearg{#1}\def\nothing{}%
\ifx\thearg\nothing
\vrule width0pt height\cellsz depth0pt\else
\hbox to 0pt{\usebox{\cell} \hss}\fi%
\vbox to \cellsz{
\vss
\hbox to \cellsz{\hss$#1$\hss}
\vss}}
\newcommand\tableau[1]{\vtop{\let\\\cr
\baselineskip -16000pt \lineskiplimit 16000pt \lineskip 0pt
\ialign{&\cellify{##}\cr#1\crcr}}}

\newcommand{\excise}[1]{}%{$\star$\textsc{#1}$\star$}

\newcommand{\scal}[1]{\langle #1 \rangle}

\renewcommand{\det}{\mathrm{det}}

\renewcommand{\lg}{\mathfrak{g}}
\title{Newell-Littlewood numbers III: eigencones and GIT-semigroups}
\author{Shiliang Gao}
\address{Dept.~of Mathematics, University of Illinois at Urbana-Champaign, Urbana, IL 61801}
\email{sgao23@illinois.edu}
\author{Gidon Orelowitz}
\address{Dept.~of Mathematics, University of Illinois at Urbana-Champaign, Urbana, IL 61801}
\email{gidono2@illinois.edu}
\author{Nicolas Ressayre}
\address{Institut Camille Jordan (ICJ), UMR CNRS 5208, Universit\'e Claude Bernard Lyon I, 43 boulevard du 11 novembre 1918, F - 69622 Villeurbanne cedex}
\email{ressayre@math.univ-lyon1.fr}
\author{Alexander Yong}
\address{Dept.~of Mathematics, University of Illinois at Urbana-Champaign, Urbana, IL 61801}
\email{ayong@illinois.edu}

\date{June 21, 2022}
\begin{document}
\pagestyle{plain}
\begin{abstract}
The \emph{Newell-Littlewood numbers} are tensor 
product multiplicities of Weyl modules for the classical groups in the
stable range.  Littlewood-Richardson
coefficients form a special case. Klyachko connected eigenvalues of sums of Hermitian matrices to the saturated LR-cone and established
defining linear inequalities. We prove analogues for the saturated NL-cone:
\begin{itemize}
\item  a 
description by \emph{Extended Horn inequalities} (as conjectured in part II of this series): using a result of King, 
this description is controlled by the saturated LR-cone and thereby recursive, just like the Horn inequalities;
\item a minimal list of 
defining linear inequalities;
\item an eigenvalue interpretation; and
\item a factorization of NL-numbers, on the boundary.
\end{itemize}
\end{abstract}

\maketitle
%\tableofcontents
%\vspace{-.3in}
\section{Introduction}\label{sec:intro}
Fix $n\in \NN:=\{1,2,3,\ldots\}$.
This is the third installment in a series \cite{GOY20a,GOY20b} about the \emph{Newell-Littlewood numbers} \cite{Newell, Littlewood} 
\begin{equation}
\label{eqn:Newell-Littlewood}
N_{\lambda,\mu,\nu}=\sum_{\alpha,\beta,\gamma} c_{\alpha,\beta}^{\lambda}
c_{\beta,\gamma}^{\mu}c_{\gamma,\alpha}^{\nu};
\end{equation}
the indices are partitions 
in 
${\Par}_n=\{(\lambda_1, \lambda_2, \ldots, \lambda_n)\in {\mathbb Z}_{\geq 0}^n:
\lambda_1\geq \lambda_2\geq \cdots \geq \lambda_n\}$.
In \eqref{eqn:Newell-Littlewood}, $c^{\lambda}_{\alpha,\beta}$ is the  
\emph{Littlewood-Richardson coefficient}.
The Littlewood-Richardson coefficients are themselves Newell-Littlewood numbers:
 if $|\nu|=|\lambda|+|\mu|$ then
$N_{\lambda,\mu,\nu}=c_{\lambda\mu}^\nu$.
The goal of this series is to establish analogues of results known for  
Littlewood-Richardson
coefficients. This paper proves NL-generalizations of breakthrough results of
Klyachko \cite{Klyachko}.

The paper \cite{GOY20a} investigated
$$
\nls(n)=\{(\lambda,\mu,\nu)\in(\Par_n)^3\,:\, N_{\lambda,\mu,\nu}> 0\}.
$$
Indeed, $\nls$ is a finitely generated semigroup \cite[Section~5.2]{GOY20a}. A good
approximation of it is the saturated semigroup:
$$
\nlss(n)=\{(\lambda,\mu,\nu)\in (\Par_n^\QQ)^3\,:\, \exists t>0 \quad N_{t\lambda,t\mu,t\nu}\neq 0\},
$$
where
$\Par_n^\QQ=\{(\lambda_1,\ldots,\lambda_n)\in\QQ^n\,:\,\lambda_1\geq\cdots\geq\lambda_n\geq
0\}$. Our main results give descriptions of $\nlss(n)$, 
including with a minimal list of defining linear inequalities.

Fix $m\in \NN$ and consider the symplectic Lie
algebra $\sp(2m,\CC)$. The irreducible $\sp(2m,\CC)$-representations $V(\lambda)$
are parametrized by their highest weight $\lambda\in\Par_m$ (see Section~\ref{sec:rootC} for details). The tensor product multiplicities
 $\mult^m_{\lambda,\mu,\nu}$ are defined by
$$
V(\lambda)\otimes V(\mu)=\sum_{\nu\in\Par_m} V(\nu)^{\oplus \mult^m_{\lambda,\mu,\nu}}.
$$ 
Since $\sp(2m,\CC)$-representations are self-dual, $\mult^m_{\lambda,\mu,\nu}$ is
symmetric in its inputs.
The supports of these multiplicities (and more generally when
$\sp(2m,\CC)$ is replaced by any semisimple Lie algebra) are of significant
interest (see, \emph{e.g.}, the survey \cite{Kumar:surveyEMS} and the references therein). Consider the finitely generated semigroup
$$
\sps(m)=\{(\lambda,\mu,\nu)\in(\Par_m)^3\,:\, \mult^m_{\lambda,\mu,\nu}> 0\},
$$
and the cone generated by it
$$
\spss(m)=\{(\lambda,\mu,\nu)\in (\Par_m^\QQ)^3\,:\, \exists t>0 \quad \mult^m_{t\lambda,t\mu,t\nu}> 0\}.
$$

For $m\geq n$, by postpending $0$'s, $\Par_n$ embeds into $\Par_m$. 
Newell-Littlewood numbers are tensor
product multiplicities for $\sp(2m,\CC)$ in the \emph{stable range} \cite[Corollary~2.5.3]{Koike}:
\begin{equation}
\forall (\lambda,\mu,\nu)\in(\Par_n)^3\quad \mbox{if $m\geq 2n$ then }
\mult^m_{\lambda,\mu,\nu}=N_{\lambda,\mu,\nu}.
\label{eq:N=m}
\end{equation}
Now, \eqref{eq:N=m} immediately implies
\begin{equation}
\label{eqn:firstobs}
\nlss(n)=\spss(m)\cap(\Par_n^\QQ)^3, \text{\ for any $m\geq 2n$.}
\end{equation}
Our first result says the relationship of $\nlss$ to $\spss$ is \emph{independent} of the stable range. 
\begin{theorem}
  \label{th:cones=}
For any $m\geq n\geq 1$,
$$
\nlss(n)=\spss(m)\cap(\Par_n^\QQ)^3.
$$
\end{theorem}

Theorem~\ref{th:cones=} has a number of consequences.
Define
$$
\lrss(n)=\{(\lambda,\mu,\nu)\in (\Par_n^\QQ)^3\,:\, \exists t>0 \quad c_{t\lambda,t\mu}^{t\nu}> 0\}.
$$
Klyachko \cite{Klyachko} 
showed that $\lrss(n)$ describes the possible eigenvalues $\lambda,\mu,\nu$ of two 
$n\times n$ Hermitian matrices $A,B,C$ (respectively) such that $A+B=C$.
Similarly, Theorem~\ref{th:cones=} 
shows that $\nlss(n)$ describes solutions to a more general eigenvalue problem;
see Section~\ref{sec:eigencone} and Proposition~\ref{prop:eigenvaluedesc}.\footnote{We remark that since $(\lambda,\mu,\nu)\in \lrss(n)$ is also in $\nlss(n)$ if and only if $|\lambda|+|\mu|=|\nu|$ (see \cite[Lemma~2.2(II)]{GOY20a}), LR-sat$(n)$ is a facet of the ${\mathfrak{sp}}_{2n}$-sat(n).}

Another major accomplishment of \cite{Klyachko} was the first proved description of $\lrss(n)$
\emph{via} linear inequalities. We have three such descriptions of
$\nlss(n)$. We now state the first of these. Set $[n]=\{1,\dots,n\}$.
For $A\subset [n]$ and $\lambda\in\Par_n$, let $\lambda_A$ be
the partition using the only parts indexed by $A$; namely,
if $A=\{i_1<\cdots<i_r\}$ then
$\lambda_A=(\lambda_{i_1},\dots,\lambda_{i_r})$.
In particular, $|\lambda_A|=\sum_{i\in A}\lambda_i$.
Using the known descriptions of $\spss(n)$ \cite{BK,GITEigen,algo}
we deduce from Theorem~\ref{th:cones=} a \emph{minimal}
list of inequalities defining $\nlss(n)$:

\begin{theorem}
  \label{th:minlistineq}
Let $(\lambda,\mu,\nu)\in(\Par_n)^3$. Then $(\lambda,\mu,\nu)\in
\nlss(n)$ if and only if
\begin{equation}
\label{Ineq:ABC}
0\leq |\lambda_A|-|\lambda_{A'}|+|\mu_B|-|\mu_{B'}|+|\nu_C|-|\nu_{C'}|
\end{equation}  
for any subsets $A,A',B,B',C,C'\subset [n]$ such that
\begin{enumerate} 
\item
$A \cap A'= B \cap B' = C \cap C' = \emptyset$;
\item \label{cond:sumr} 
$|A|+|A'|=|B|+|B'|=|C|+|C'|=|A'|+|B'|+|C'|=:r$;
\item \label{cond:c1}
$c_{\tau^0(A,A')^{\vee[(2n-2r)^r]}\;
    \tau^0(B,B')^{\vee[(2n-2r)^r]}}^{\tau^0(C,C')}=
c_{\tau^2(A,A')^{\vee[r^r]}\; \tau^2(B,B')^{\vee[r^r]}}^{\tau^2(C,C')}=1$.
\end{enumerate}
Moreover, this list of inequalities is irredundant.
\end{theorem}

The definition of the
partitions occurring in condition~\eqref{cond:c1} is in
Section~\ref{sec:SchubSp}.

The proofs of Theorems~\ref{th:cones=} and \ref{th:minlistineq} use
ideas of P.~Belkale-S.~Kumar \cite{BK} on their deformation of the cup product on flag manifolds, as well as the third author's work on
GIT-semigroups/cones \cite{GITEigen, algo}. We interpret $\nlss(n)$
from the latter perspective in Section~\ref{sec:truncated}
(see Proposition~\ref{prop:ineqP0}) by study of the \emph{truncated tensor
cone}. 
Our argument requires us to generalize \cite[Theorem~28]{BK} \cite[Theorem~B]{GITEigen}
(recapitulated here as Theorem~\ref{th:BK}); see Theorem~\ref{th:BKP}.
As an application, we obtain Theorem~\ref{th:reduction} below, which is 
a factorization of the
NL-coefficients on the boundary of $\nlss(n)$.
Let $\lambda\in\Par_n$ and $A,A'\subset[n]$.
Write $A'=\{i'_1<\cdots<i'_s\}$ and $A=\{i_1<\cdots<i_t\}$ and set
\[
\lambda_{A,A'}=(\lambda_{i'_1},\dots,\lambda_{i'_s},-\lambda_{i_t},\dots,-\lambda_{i_1})
\text{\ \ and \ $
\lambda^{A,A'}=\lambda_{[n]-(A\cup A')}$, \emph{etc}.}\]

\begin{theorem}
\label{th:reduction}
  Let $A,A',B,B',C,C'\subset [n]$ such that
\begin{enumerate} 
\item
$A \cap A'= B \cap B' = C \cap C' = \emptyset$;
\item $|A|+|A'|=|B|+|B'|=|C|+|C'|=|A'|+|B'|+|C'|=:r$;
\item %\label{cond:c1}
$c_{\tau^0(A,A')^{\vee[(2n-2r)^r]}\;
    \tau^0(B,B')^{\vee[(2n-2r)^r]}}^{\tau^0(C,C')}=
c_{\tau^2(A,A')^{\vee[r^r]}\; \tau^2(B,B')^{\vee[r^r]}}^{\tau^2(C,C')}=1$,
\end{enumerate}
as in Theorem~\ref{th:minlistineq}.
For $(\lambda,\mu,\nu)\in(\Par_n)^3$ such that
\begin{equation}
\label{eq:ABC}
0=|\lambda_A|-|\lambda_{A'}|+|\mu_B|-|\mu_{B'}|+|\nu_C|-|\nu_{C'}|,
\end{equation}  
\begin{equation}
  \label{eq:NLred}
  N_{\lambda,\mu,\nu}=c_{\lambda_{A,A'},\mu_{B,B'}}^{\nu_{C,C'}^*} N_{\lambda^{A,A'},\mu^{B,B'},\nu^{C,C'}}.
\end{equation}
\end{theorem}
Theorem~\ref{th:reduction} is an analogous to 
\cite[Theorem~7.4]{DW:comb} and \cite[Theorem~1.4]{KTT:factorLR} 
for $c_{\lambda,\mu}^{\nu}$. 

Knutson-Tao's celebrated \emph{Saturation Theorem} \cite{KT:saturation} proves,
\emph{inter alia}, that $\lrs(n)$ is  described by Horn's inequalities (see, \emph{e.g.}, Fulton's survey
\cite{Fulton:survey}). 
This posits a generalization:
\begin{conjecture}[{NL-Saturation \cite[Conjecture~5.5]{GOY20a}}]
  \label{conj:saturation}
Let $(\lambda,\mu,\nu)\in(\Par_n)^3$. Then 
$N_{\lambda,\mu,\nu}\neq 0$ if and only if $|\lambda|+|\mu|+|\nu|$ is
even and there exists
$t>0$ such that $N_{t\lambda,t\mu,t\nu}\neq 0$.
\end{conjecture}
Theorem~\ref{th:minlistineq} permits us to prove 
Conjecture~\ref{conj:saturation} for  $n\leq 5$, by computer-aided calculation of Hilbert bases;
see Section~\ref{sec:apptosat5}. This is the strongest evidence of the conjecture
to date; previously, \cite[Corollary~5.16]{GOY20a} proved the $n=2$ case, by combinatorial
reasoning.

Let $\lambda_1,\dots,\lambda_s\in \Par_n$ for $s\geq 3$. 
Treat the indices $1,\dots,s$ as elements of $\ZZ/s\ZZ$. 
We introduce the \emph{multiple Newell-Littlewood number} as
\begin{equation}
  \label{eq:defmNL}
  N_{\lambda_1,\dots,\lambda_s}=\sum_{(\alpha_1,\dots,\alpha_s)\in(\Par_n)^{s}}\prod_{i\in\ZZ/s\ZZ}
c_{\alpha_i\,\alpha_{i+1}}^{\lambda_i}. 
\end{equation}
When $s=3$, we recover the Newell-Littlewood
numbers.\footnote{$N_{\lambda_1,\dots,\lambda_s}$ also has a uniform
  representation-theoretic interpretation. Discussion may appear
  elsewhere.}  
Consider the associated semi-group and cone:
$$
\operatorname{NL}^s\!\operatorname{-semigroup}
  (n)=\{(\lambda_1,\dots,\lambda_s)\in(\Par_n)^s\,:\,
  N_{\lambda_1,\dots,\lambda_s}> 0\}, 
$$
and
$$
\operatorname{NL}^s\!\operatorname{-sat}(n)=\{(\lambda_1,\dots,\lambda_s)\in (\Par_n^\QQ)^s\,:\, \exists t>0 \quad N_{t\lambda_1,\dots,t\lambda_s}\neq 0\}.
$$
To $A=\{i_1<\cdots<i_r\}\subset[n]$, associate the partition
\begin{equation}
\label{taudef}
\tau(A)=(i_r-r\geq\cdots\geq i_1-1).
\end{equation}

The Horn
inequalities for $\lrss(n)$ are recursive, as they depend on
$\lrss(n')$ for $n'<n$ (see \cite{Fulton:survey}).
Theorem~\ref{th:minlistineq}  is not recursive. 
However, our next result describes the cone $\nlss(n)$ by inequalities
depending on $\operatorname{NL}^6\!\operatorname{-sat}(n')$ for $n'\leq
  n$.

\begin{theorem}
  \label{th:GOYconj}
Let $(\lambda,\mu,\nu)\in(\Par_n^{\mathbb Q})^3$. Then $(\lambda,\mu,\nu)\in
\nlss(n)$ if and only if
\begin{equation}
\label{Ineq:EH}
0\leq |\lambda_A|-|\lambda_{A'}|+|\mu_B|-|\mu_{B'}|+|\nu_C|-|\nu_{C'}|
\end{equation}  
for any subsets $A,A',B,B',C,C'\subset [n]$ such that
\begin{enumerate}
\item $A \cap A'= B \cap B' = C \cap C' = \emptyset$;
\item $|A|+|A'|=|B|+|B'|=|C|+|C'|=|A'|+|B'|+|C'|=:r$;
\item \label{cond:ind6} 
$(\tau(A), \tau(C'), \tau(B), \tau(A'), \tau(C), \tau(B'))\in \operatorname{NL}^6\!\operatorname{-sat}(r)$. 
\end{enumerate}
\end{theorem}
 
This result is proved in Section~\ref{sec:relationtoGOY}.  
By a result of R.~King \cite{King71} (see also \cite{HTWbranching}),
each six-fold NL-coefficient is a particular Littlewood-Richardson coefficient
(see Section~\ref{sec:relationtoGOY} for details).
Consequently, condition \eqref{cond:ind6} is equivalent to checking if some
explicitly determined triple of partitions is in $\lrss(2r)$. Since $\lrss(2r)$ is described by the Horn inequalities,
we thereby obtain a description of $\nlss(n)$ only involving inequalities and no tensor product multiplicities. It is in this sense that Theorem~\ref{th:GOYconj} is of the same spirit as Horn's original inequalities.

Just as the proof of A.~Horn's conjecture depends on A.~Knutson-T.~Tao's saturation theorem, our
proof of Theorem~\ref{th:GOYconj} uses this consequence
of R.~King's result (see Section~\ref{sec:GLstable}):

\begin{proposition}[Six-fold NL saturation]\label{prop:satNL6}
  Let $\lambda_1,\dots,\lambda_6\in \Par_n$. If there
  exists a positive integer $t$ such that
  $N_{t\lambda_1,\dots,t\lambda_6}\neq 0$ then
  $N_{\lambda_1,\dots,\lambda_6}\neq 0$.
\end{proposition}

In \cite[Conjecture~1.4]{GOY20b}, a conjectural 
description of $\nlss(n)$ was given. That conjecture subsumes 
both Conjecture~\ref{conj:saturation} and a description of $\nlss$ using \emph{extended Horn inequalities} \cite[Definition~1.2]{GOY20b}. 
Theorem~\ref{th:GOYconj} proves the latter part of the conjecture.

\section{Generalities on the tensor cones}

\subsection{Finitely generated semigroups} A subset
$\Gamma\subseteq\ZZ^n$ is a {\it semigroup} if $\vec 0\in \Gamma$ and
$\Gamma$ is closed under addition. A finitely generated semigroup
$\Gamma$ \emph{generates} a closed convex polyhedral cone $\Gamma_\QQ\subseteq \QQ^n$:
$$
\Gamma_\QQ=\{x\in\QQ^n\;:\;\exists t\in\ZZ_{>0}\quad tx\in\Gamma\}.
$$
The subgroup of $\ZZ^n$ generated by $\Gamma$ is
$$
\Gamma_\ZZ=\{x-y\,:\,x,y\in\Gamma\}.
$$
The semigroup $\Gamma$ is {\it saturated} if 
$\Gamma=\Gamma_\ZZ\cap\Gamma_\QQ$.

\subsection{GIT-semigroups}\label{sec:GITsemi}

We recall the GIT-perspective of \cite{GITEigen}. Let $G$ be a complex reductive group acting on an irreducible projective variety
$X$. Let $\Pic^G(X)$ be the group of $G$-linearized line bundles.
 Given $\Li\in \Pic^G(X)$, let $\Ho^0(X,\Li)$ be the
space of sections of $\Li$; it is a $G$-module. Let $\Ho^0(X,\Li)^G$
be the subspace of invariant sections. Define
$$
\gits(G,X)=\{\Li\in \Pic^G(X)\,:\, \Ho^0(X,\Li)^G\neq\{0\}\}.
$$
This is a semigroup since $X$ being irreducible says the product of two
nonzero $G$-invariant sections is a nonzero $G$-invariant section.
The saturated version of it is 
$$
\gitss(G,X)=\{\Li\in \Pic^G(X)\otimes\QQ\,:\, \exists t>0 \quad \Ho^0(X,\Li^{\otimes t})^G\neq\{0\}\}.
$$

\subsection{The tensor semigroup}\label{subsec:tensorsemi}
Let $\lg$ be a semisimple complex Lie algebra, with fixed Borel
subalgebra $\lb$ and Cartan subalgebra $\lt\subset\lb$. 
Denote by $\Lambda^+(\lg)\subset\lt^*$ the semigroup of the dominant
weights.
It is contained in the weight lattice $\Lambda(\lg)\simeq\ZZ^r$, where $r$ is the rank of $\lg$.
Given $\lambda\in \Lambda^+(\lg)$, denote by $V_\lg(\lambda)$ (or
simply $V(\lambda)$) the irreducible representation of $\lg$ with
highest weight $\lambda$. Let $V(\lambda)^*$ be the dual representation.
Consider the semigroup
$$
\lgs=\{(\lambda,\mu,\nu)\in(\Lambda^+(\lg))^3\;:\;V(\nu)^*\subset
V(\lambda)\otimes V(\mu)\},
$$
and the generated cone $\lgss$ in $(\Lambda(\lg)\otimes \QQ)^3$.
When $\lg=\sp(2m,\CC)$ we have $V(\nu)^*\simeq V(\nu)$ and $\lgs$ is what we
denoted by $\sp$-semigroup$(m)$ in the introduction.
The set $\lgs$ spans the rational vector space
$(\Lambda(\lg)\otimes \QQ)^3$, or equivalently, the cone $\lgss$ has
nonempty interior. 
The group $(\lgs)_\ZZ$ is well-known (see,
\emph{e.g.}, \cite[Theorem~1.4]{PRexample}):
$$
(\lgs)_\ZZ=\{(\lambda,\mu,\nu)\in (\Lambda(\lg))^3\;:\;
\lambda+\mu+\nu\in\Lambda_R(\lg)\},
$$
where $\Lambda_R(\lg)$ is the root lattice of $\lg$.

We now interpret $\lgs$ in terms of Section~\ref{sec:GITsemi}. 
Consider the semisimple, simply-connected algebraic group $G$ with Lie algebra $\lg$.
Denote by $B$ and $T$ the connected subgroups of $G$ with Lie algebras
$\lb$ and $\lt$ respectively.
The character groups $X(B)=X(T)=\Lambda(\lg)$ of $B$ and $T$ coincide. For
$\lambda\in X(T)$, $\Li_\lambda$ is the unique
$G$-linearized line bundle on the flag variety $G/B$ such that $B$
acts on the fiber over $B/B$ with weight $-\lambda$.

Assume $X=(G/B)^3$. 
Then  $\Pic^G(X)$ identifies with $X(T)^3$. 
For $(\lambda,\mu,\nu)\in X(T)^3$, define $\Li_{(\lambda,\mu,\nu)}$. 
By the Borel-Weil Theorem, 
\begin{equation}
\label{eqn:April30BorelWeil}
\Ho^0(X,\Li_{(\lambda,\mu,\nu)})=V(\lambda)^*\otimes V(\mu)^*\otimes V(\nu)^*.
\end{equation}
In particular, $\gits(G,X)\simeq \lgs$.

Given three parabolic subgroups $P,Q$ and $R$ containing $B$, 
we consider more generally $X=G/P\times G/Q\times G/R$. 
Then $\Pic^G(X)$ identifies with $X(P)\times X(Q)\times X(R)$ which is
a subgroup of $X(T)^3$. Moreover,
$$
\gits(G,X)=\gits(G,(G/B)^3)\cap(X(P)\times X(Q)\times X(R)).
$$
\subsection{Schubert calculus}

We need notation for the cohomology ring
$\Ho^*(G/P,\ZZ)$; $P\supset B$ being a parabolic subgroup. 
Let $W$ (resp. $W_P$) be the Weyl group of $G$ (resp. $P$).
Let $\ell\,:\,W\longto\NN$ be the \emph{Coxeter length}, defined with respect to the
simple reflections determined by the choice of $B$.
Let $W^P$ be the minimal length representatives of the cosets
in $W/W_P$.
  
For a closed irreducible subvariety $Z\subset G/P$, let $[Z]$ be
its class in $\Ho^*(G/P,\ZZ)$, of degree $2(\dim(G/P)-\dim(Z))$.
For $v \in W^P$, set 
\[\tau_v=[\overline{BvP/P}]\] ($\dim(\overline{BvP/P})=\ell(v)$).
Then
$$
\Ho^*(G/P,\ZZ)=\bigoplus_{v\in W^P}\ZZ\tau_v.
$$
Let $w_0$ be the longest element of $W$
and $w_{0,P}$ be the longest element of $W_P$. 
Set $v^\vee=w_0vw_{0,P}$ and $\tau_v=\tau^{v^\vee}$; $\tau^v$ and $\tau_v$ are Poincar\'e dual.

Let $\rho$ be the half sum of the positive roots of $G$.
To any one-parameter subgroup $\tau:{\mathbb C}^*\to T$, associate the
parabolic subgroup (see \cite{GIT})
$$
P(\tau)=\{g\in G\,:\,\lim_{t\to 0}\tau(t)g\tau(t\inv) \mbox{ exists}\}.
$$
Fix such a $\tau$ such that $P=P(\tau)$.

For $v\in W^P$, define the \emph{BK-degree} of $\tau^v\in \Ho^*(G/P,\ZZ)$ to be
\[
\BKdeg(\tau^v):=\scal{v\inv(\rho)-\rho,\tau}
.\]
Let $v_1,v_2$ and $v_3$ in $W^P$. By \cite[Proposition 17]{BK}, if
$\tau^{v_3}$ appears in
the product $\tau^{v_1} \cdot \tau^{v_2}$ then 
\begin{equation}
  \label{eq:13}
  \BKdeg(\tau^{v_3})\leq \BKdeg(\tau^{v_1})+\BKdeg(\tau^{v_2}).
\end{equation}
In other words, the BK-degree filters the cohomology ring. 
Let $\odot_0$ denote the associated graded product on $\Ho^*(G/P,\ZZ)$.

\subsection{Well-covering pairs}
In \cite{GITEigen}, $\gitss(G,X)$ is described in terms of \emph{well-covering pairs}. 
When $X=(G/B)^3$, it recovers the description made by
Belkale-Kumar \cite{BK}. We now discuss the case when
$X=G/P\times G/Q\times G/R$ is the product of three partial flag varieties of $G$.
 
Let $\tau$ be a dominant one-parameter subgroup of $T$. 
The centralizer $G^\tau$ of the image of $\tau$ in $G$ is a Levi
subgroup. 
Moreover, 
$
P(\tau)
$
is the parabolic subgroup generated by $B$ and $G^\tau$.
Let $C$ be an irreducible component of the fixed set $X^\tau$ of
$\tau$ in $X$. 
It is well-known that $C$ is the $(G^\tau)^3$-orbit of some $T$-fixed
point:
\begin{equation}
C=G^\tau u\inv P/P\times G^\tau v\inv Q/Q \times G^\tau w\inv
R/R,\label{eq:Cw}
\end{equation}
with $u\in {W_P}\backslash W/W_{P(\tau)}$ and similarly for $v$ and
$w$.
Set
$$
C^+=P(\tau) u\inv P/P\times P(\tau) v\inv  Q/Q \times P(\tau) w\inv  R/R.
$$
Then the closure of $C^+$ is a Schubert variety (for $G^3$) in $X$.
By \cite[Proposition~11]{GITEigen}, the pair $(C,\tau)$ is
{\it well-covering} if and only if
\begin{equation}
[\overline{PuP(\tau)/P(\tau)}]\odot_0
[\overline{QvP(\tau)/P(\tau)}]\odot_0
[\overline{RwP(\tau)/P(\tau)}]=[pt]
\in \Ho^*(G/P(\tau),\ZZ).
\label{eq:defwc}
\end{equation}
It is said to be {\it dominant} if
\begin{equation}
[\overline{PuP(\tau)/P(\tau)}]\cdot
[\overline{QvP(\tau)/P(\tau)}]\cdot
[\overline{RwP(\tau)/P(\tau)}]\neq 0
\in \Ho^*(G/P(\tau),\ZZ).
\label{eq:defdom}
\end{equation}
In this paper, the reader can take these characterizations as 
definitions of  well-covering and dominant pairs.
They are used in \cite{GITEigen} to produce inequalities for the
GIT-cones:

\begin{proposition}
  \label{prop:coneineqplus}
Let $(\lambda,\mu,\nu)\in X(P)\times X(Q)\times X(R)$ be dominant such that
$\Li_{(\lambda,\mu,\nu)}\in\gitss(G,X)$.
Let  $(C,\tau)$ be a dominant pair.
Then,
\begin{equation}
  \label{eq:domineq}
  \scal{u\tau,\lambda}+\scal{v\tau,\mu}+\scal{w\tau,\nu}\leq 0.
\end{equation}
Here $(u,v,w)$ are determined by $C$ and equation~\eqref{eq:Cw}.
\end{proposition}

In Proposition~\ref{prop:coneineqplus} and below, $\scal{\cdot,\cdot}$
denotes the pairing between one parameter subgroups and characters of $T$.
Among the inequalities given by Proposition~\ref{prop:coneineqplus},
those associated to well-covering pairs are sufficient to define the
GIT-cone
(see \cite[Proposition~4]{GITEigen}):

\begin{proposition}\label{prop:cone1}
  Let $(\lambda,\mu,\nu)\in X(P)\times X(Q)\times X(R)$ be dominant. Then,
\[\Li_{(\lambda,\mu,\nu)}\in \gitss(G,X)\] 
if and only if for any
dominant one-parameter subgroup $\tau$ of $T$ and any well-covering
pair $(C,\tau)$,
\begin{equation}
  \label{eq:wcineq}
  \scal{u\tau,\lambda}+\scal{v\tau,\mu}+\scal{w\tau,\nu}\leq 0.
\end{equation}
Here $(u,v,w)$ are determined by $C$ and equation~\eqref{eq:Cw}.\footnote{Proposition~\ref{prop:cone1} implies Proposition~\ref{prop:coneineqplus} is also a characterization of $\gitss(G,X)$, but we have stated the weaker form of the latter to emphasize it is an easier result.}
\end{proposition}

In the case $P=Q=R=B$, there is a more precise
statement.  The fact that the inequalites define the cone is 
Belkale-Kumar \cite[Theorem~28]{BK}. 
The irredundancy is \cite[Theorem~B]{GITEigen}.
Let $\alpha$ be a simple root of
$G$. Denote by $P^\alpha$ the associated
maximal parabolic subgroup of $G$ containing $B$. 
Denote by $\varpi_{\alpha^\vee}$ the associated fundamental
one-parameter subgroup of $T$ characterized by
$\scal{\varpi_{\alpha^\vee},\beta}=\delta_\alpha^\beta$ for any simple
root $\beta$.

\begin{theorem}[{\cite[Theorem~28]{BK},\cite[Theorem~B]{GITEigen}}]\label{th:BK}
Here $X=(G/B)^3$.
   Let $(\lambda,\mu,\nu)\in X(T)^3$ be dominant. 
Then,
$\Li_{(\lambda,\mu,\nu)}\in\gitss(G,X)$ if and only if for any
simple root $\alpha$, for any $u,v,w$ in $W^{P^\alpha}$ such that
\begin{equation}
  \label{eq:BKcohomGB}
  [\overline{BuP^\alpha/P^\alpha}]\odot_0
[\overline{BvP^\alpha/P^\alpha}]\odot_0
[\overline{BwP^\alpha/P^\alpha}]=[pt]
\in \Ho^*(G/P^\alpha,\ZZ),
\end{equation}
\begin{equation}
  \label{eq:BKineqGB}
  \scal{u\varpi_{\alpha^\vee},\lambda}+\scal{v\varpi_{\alpha^\vee},\mu}+\scal{w
    \varpi_{\alpha^\vee},\nu}\leq 0.
\end{equation}
Moreover, this list of inequalities is irredundant. 
\end{theorem}

Theorem~\ref{th:BK} can be obtained from
Proposition~\ref{prop:cone1} by showing that it is sufficient to
consider the one-parameter subgroups $\tau$ equal to
$\varpi_{\alpha^\vee}$ for some simple root $\alpha$.
See the proof of Theorem~\ref{th:BKP} below for a similar argument.

\subsection{The eigencone}
\label{sec:eigencone}
A relationship between $\lgss$ and projections of coadjoint orbits
was discovered by Heckman \cite{Heck}. Theorem~\ref{th:KN} below 
interprets $\lgss$ in terms of 
eigenvalues.

Fix a maximal compact subgroup $U$ of $G$ such that $T\cap U$ is a Cartan subgroup of $U$.
Let $\lu$ and $\lt$ denote the Lie algebras of $U$ and $T$ respectively.
Let $\lt^+$ be the Weyl chamber of $\lt$ corresponding to $B$.
Let $\sqrt{-1}$ denote the usual complex number.
It is well-known that $\sqrt{-1}\lt^+$ is contained in $\lu$ and that the map
\begin{equation}
\begin{array}{ccc}
  \lt^+&\longto&\lu/U\\
\xi&\longmapsto&U.(\sqrt{-1}\xi)
\end{array}
\label{eq:homeospec}
\end{equation}
is a homeomorphism.
Here $U$ acts on $\lu$ by the adjoint action.
Consider the set
$$
\Gamma(U):=\{(\xi,\zeta,\eta)\in(\lt^+)^3\,:\,
U.(\sqrt{-1}\xi)+U.(\sqrt{-1}\zeta)+U.(\sqrt{-1}\eta)\ni 0\}.
$$

Let $\lu^*$ (resp. $\lt^*$) denote the dual (resp. complex dual) of $\lu$ (resp. $\lt$).
Let $\lt^{*+}$ denote the dominant chamber of $\lt^*$ corresponding to $B$.
By taking the tangent map at the identity, one can embed $X(T)^+$ in $\lt^{*+}$.
Note that this embedding induces a rational structure on the complex vector space $\lt^*$.
Moreover it allows to embed the tensor cone $\lgss$  in $(\lt^{*+})^3$.

The Cartan-Killing form allows to identify $\lt^+$ and $\lt^{*+}$.
In particular $\Gamma(U)$ also embeds in $(\lt^{*+})^3$; the so
obtained subset of $(\lt^{*+})^3$ is
denoted by $\tilde\Gamma(U)$ to avoid any confusion. The following result is
well-known; see \emph{e.g.}, \cite[Theorem~5]{Kumar:SurveyTransformation} 
and the references therein.

\begin{theorem}
\label{th:KN}
The set $\Gamma(U)$ is a closed convex  polyhedral cone. 
Moreover, $\lgss$ is the set of  the rational points in $\tilde\Gamma(U)$. 
\end{theorem}

\section{The case of the symplectic group}
\subsection{The root system of type $C$}
\label{sec:rootC}

Let $V={\mathbb C}^{2n}$
with the standard basis $(\vec e_1,\vec e_2,\ldots,\vec e_{2n})$. Let $J_n$ be the $n\times n$ ``anti-diagonal'' identity matrix and define
a skew-symmetric bilinear form $\omega(\bullet,\bullet): V\times V\to {\mathbb C}$ using the block matrix 
$\Omega:=\left[\begin{matrix} 0 & J_n \\ -J_n & 0 \end{matrix}\right]$. By definition, the \emph{symplectic group} 
$G=\Sp(2n,\CC)$ is the group of
automorphisms of $V$ that preserve this bilinear form.

  Given a $n\times n$ matrix $A=(A_{ij})_{1\leq i,j\leq n}$ define
  ${}^TA$ by $({}^TA)=A_{n+1-j\,n+1-i}$, obtained from $A$ by
  reflection across the antidiagonal. The Lie
  algebra $\sp(2n,\CC)$ is the set of matrices $M\in {\sf Mat}_{2n\times 2n}(\CC)$
  such that ${}^tM\Omega+\Omega M=0$; namely,
  \begin{equation}
    \label{eq:sp2n}
    \sp(2n,\CC)=
    \left\{
    \begin{pmatrix}
      A&B\\
C&-{}^TA
    \end{pmatrix}
:
\begin{array}{l}
A,B,C \mbox{ of size }n\times n,\\
{}^TB=B
\mbox{ and } {}^TC=C
\end{array}
\right\}
  \end{equation}
which has complex dimension $2n^2+n$.
The Lie algebra $\lu(2n,\CC)$ of the unitary group $U(2n,\CC)$ is the
set of anti-Hermitian matrices. Thus, \eqref{eq:sp2n} gives
\begin{equation}
\sp(2n,\CC)\cap \lu(2n,\CC)=
\left\{
    \begin{pmatrix}
      A&B\\
 -{}^t\bar B&-{}^TA
    \end{pmatrix} : {}^t\bar A=-A
\mbox{ and } {}^TB=B
\right\}
\label{eq:sp2nu}
\end{equation}
which has real dimension $2n^2+n$.
As a consequence, $U(2n)\cap \Sp(2n,\CC)$ is a maximal compact subgroup
of $\Sp(2n,\CC)$.

 Let $B$ be the Borel subgroup of upper triangular matrices in $G$. 
Let
\[T=\{\diag(t_1,\ldots,t_{n},t_{n}^{-1},\ldots,t_1^{-1})\,:\,t_i\in\CC^*\}\]
be the maximal torus contained in $B$.
For $i\in [1,n]$, let $\varepsilon_i$ denote the character of $T$ that maps 
$\diag(t_1,\ldots,t_{n},t_{n}^{-1},\ldots,t_1^{-1})$ to $t_i$; then 
$X(T)=\oplus_{i=1}^n\ZZ\varepsilon_i$.
Here
$$
\begin{array}{l}
  \Phi^+=\{\varepsilon_i\pm\varepsilon_j\,:\,1\leq i<j\leq n\}\cup 
\{2\varepsilon_i\,:\,1\leq i\leq n\}, \\
\Delta=\{\alpha_1=\varepsilon_1-\varepsilon_2,\,\alpha_2=\varepsilon_2-\varepsilon_3,\ldots,\,
\alpha_{n-1}=\varepsilon_{n-1}-\varepsilon_{n},\,\alpha_{n}=2\varepsilon_{n}\},
  {\rm\ and}\\
X(T)^+=\{\sum_{i=1}^{n}\lambda_i \varepsilon_i\,:\,
  \lambda_1\geq\cdots \geq \lambda_{n}\geq 0\}=\Par_n.
\end{array}
$$
For $i\in [1,2n]$, set $\overline{i}=2n+1-i$.
The Weyl group $W$ of $G$ may be identified with a subgroup of the Weyl group $S_{2n}$ of
$\SL(V)$. More precisely,
$$
W=\{w\in S_{2n}\,:\,w(\overline{i})=\overline{w(i)} \ \ \forall i\in [1,2n]\}.
$$

Observe that $T\cap U(2n,\CC)$ has real dimension $n$ and is a maximal
torus of  $U(2n)\cap \Sp(2n,\CC)$. 
The bijection~\eqref{eq:homeospec} implies that any matrix $M_1=\begin{pmatrix}
      A&B\\
 -{}^t\bar B&-{}^TA
    \end{pmatrix}$ in $\sp(2n,\CC)\cap \lu(2n,\CC)$
    (see~\eqref{eq:sp2nu}) is diagonalizable with eigenvalues in
    $\sqrt{-1}\RR$. Moreover, ordering the eigenvalues by
    nonincreasing order, we get 
    $$
\lambda(\sqrt{-1}M_1)\in\{(\lambda_1\geq\cdots\geq\lambda_n\geq
-\lambda_n\geq\cdots\geq-\lambda_1)\,:\,\lambda_i\in\RR\}.
$$ 
For
$\lambda=(\lambda_1,\dots,\lambda_n)\in\Par_n$, 
set $\hat\lambda=(\lambda_1,\dots,\lambda_n,-\lambda_n,\dots,-\lambda_1)$. 
Now, Theorems~\ref{th:cones=} with $m=n$ and Theorem~\ref{th:KN} give an
interpretation of $\nlss(n)$ in terms of eigenvalues:

\begin{proposition}
\label{prop:eigenvaluedesc}
  Let $\lambda,\mu,\nu\in\Par_n$. Then $(\lambda,\mu,\nu)\in \nlss(n)$ if and only if there exist three matrices $M_1,M_2,M_3\in
  \sp(2n,\CC)\cap \lu(2n,\CC)$ such that $M_1+M_2+M_3=0$ and
$$
(\hat\lambda,\hat\mu,\hat\nu)=(\lambda(\sqrt{-1}M_1),
\lambda(\sqrt{-1}M_2), \lambda(\sqrt{-1}M_3)).
$$
\end{proposition}

\subsection{Isotropic Grassmannians and Schubert classes}
\label{sec:SchubSp}

Our reference for this subsection is  \cite[Section~5]{algo}.
For $r=1,\dots,n$, the one-parameter subgroup $\varpi_{\alpha_r^\vee}$
is given by 
\[\varpi_{\alpha_r^\vee}
(t)=\diag(t,\dots,t,1,\dots,1,t\inv,\dots,t\inv),\] 
where $t$ and
$t\inv$ occur $r$ times. 

A subspace $W\subseteq V$ is \emph{isotropic} if for all $\vec v, \vec v'\in W$, $\omega(\vec v, \vec v')=0$.  Given an $r$-subset  $I\subset[2n]$, we set $F_I=\Span(\vec e_i:i\in I)$. Clearly, 
$F_I$ is isotropic if and only if $I\cap\bar I=\emptyset$, where 
$\bar I=\{\bar i\,:\,i\in I\}$. Now, $P^{\alpha_r}$ is the stabilizer of the
isotropic subspace $F_{\{1,\dots,r\}}$. Thus, 
$G/P^{\alpha_r}=\Gr_\omega(r,2n)$ is the \emph{Grassmannian of isotropic
$r$-dimensional vector subspaces of $V$}. 

Let $\Part(r,2n)$ denote the set of subsets of $\{1,\dots,2n\}$ with
$r$ elements.
Set 
$$
\Schub(\Gr_\omega(r,2n)):=\{I\in\Part(r,2n)\,:\, I\cap\overline{I}=\emptyset\}.
$$
If $I=\{i_1<\cdots<i_r\}\in \Schub(\Gr_\omega(r,2n))$,  let
$i_{\overline{k}}:=\overline{i_k}$ for $k\in [r]$, and  
$\{i_{r+1}<\cdots<i_{\overline{r+1}}\}=[2n]-(I\cup \overline{I})$. 
Therefore $w_I=(i_1,\dots,i_{2n})\in S_{2n}$ is 
the element of $W^{P^{\alpha_r}}$ corresponding to $F_I$; that is, $F_I=w_I P^{\alpha_r}/P^{\alpha_r}$.

Set
$$
\Schub'(\Gr_\omega(r,2n)):=
\left \{
  \begin{array}{ll}
(A,A')\,:\, &A\in \Part(a,n),\,
A'\in \Part(a',n)\mbox{ for some $a$ and $a'$ s.t. }\\
&a+a'=r\mbox{
  and } A\cap A'=\emptyset
  \end{array}
\right \}.
$$
This map is a bijection: 
\begin{equation}
\begin{array}{ccl}
  \Schub(\Gr_\omega(r,2n))&\longto&\Schub'(\Gr_\omega(r,2n))\\
I&\longmapsto&(\bar I\cap[n],I\cap[n]).
\end{array}\label{eq:bijIA}
\end{equation}
Recall from the introduction the definition of $\tau(I)$ and hence
$\tau(A)$ and $\tau(A')$. The relationship between these three partitions
is depicted in Figure~\ref{fig:tauIA}.

\begin{figure}
    \begin{center}
    \begin{tikzpicture}[scale=0.7]
        \draw[black,thick] (0,0) -- (9,0) -- (9,5) -- (0,5) -- (0,0);
        \draw[black, thick] (4,0) -- (4,5);
        \draw[black, thick] (0,3) -- (9,3);   
        \draw (0, 1.5) node[left]{$a'$};
        \draw (0, 4) node[left]{$a$};
        \draw (4, 5) node[above]{$2n-r$};
        \draw (2, 0) node[below]{$n-a'$};
        \draw (6.5, 0) node[below]{$n-a$};
        \draw[black,fill=gray!30] (0,0) -- (1,0) -- (1,0.5) -- (2.5,0.5) --
        (2.5,2.5) --  (3,2.5) -- (3,3) -- (4,3) -- (0,3) -- (0,0);
 
        \draw[black,fill=gray!30] (4,3) -- (4,3.5) -- (6,3.5) -- (6,4) -- (8,4) --
        (8,5) -- (9,5) -- (9,3) -- (4,3);
\fill [pattern=north east lines]  (0,0) -- (1,0) -- (1,0.5) -- (2.5,0.5) --
        (2.5,2.5) --  (3,2.5) -- (3,3) -- (4,3) -- (4,3.5) -- (6,3.5) -- (6,4) -- (8,4) --
        (8,5) -- (9,5) -- (0,5)  -- (0,0);
        \draw (1.5,2) node[fill=white]{$\tau(A')$};
        \draw (7,3.5) node[fill=white]{$\tau(A)$};

\draw[black,fill=gray!30] (11,4) -- (11,5) -- (12,5) -- (12,4) -- cycle;
\draw  (12,4.5)  node[right] {$\tau(A)$ rotated $180^\circ$};
\draw (11,2.5) -- (11,3.5) -- (12,3.5) -- (12,2.5) -- cycle;
\fill [pattern=north east lines] (11,2.5) -- (11,3.5) -- (12,3.5) -- (12,2.5) -- cycle;
\draw  (12,3)  node[right] {$\tau(I)$};
\draw[black,fill=gray!30] (11,1) -- (11,2) -- (12,2) -- (12,1) --
cycle;
\fill [pattern=north east lines] (11,1) -- (11,2) -- (12,2) -- (12,1) --
cycle;
\draw  (12,1.5)  node[right] {$\tau(A')$};
    \end{tikzpicture}
\end{center}
    \caption{$\tau(I)$, $\tau(A)$ and $\tau(A')$.}
    \label{fig:tauIA}
\end{figure}
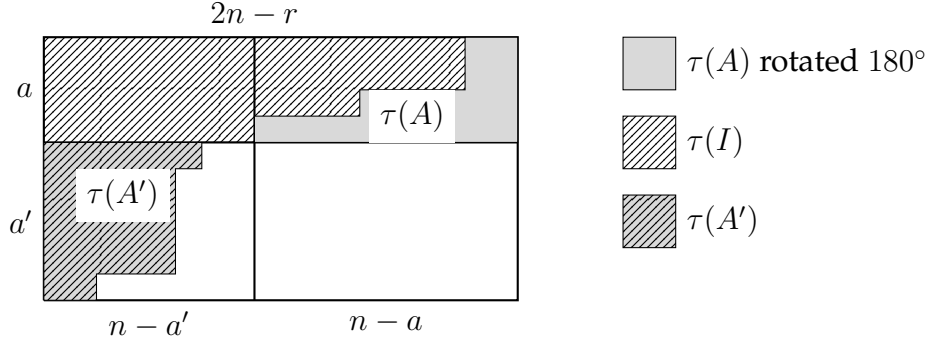

\begin{definition}
For subsets $X\subseteq Y \subseteq [2n]$, define
\begin{align*}
    \chi_{(X,Y)}:&[2n] \rightarrow  \{0,1,2\}\\
    &i \mapsto  \begin{cases} 1 & \text{if }i\in X  \\ 
    2 & \text{if } i\in Y\setminus X\\
    0 & \text{if } i \notin Y
    \end{cases}
\end{align*}
\end{definition}

Treat $\chi_{(X,Y)}$ as a word of length $2n$ with letters in
$\{0,1,2\}$. Removing all the $2$'s in the word, we obtain the
characteristic function $\chi_{(X,Y)}^{2}$ of some subset $I^2$ of
$[2n-|Y|+|X|]$
with $|X|$ elements. Similarly, removing all $0$'s in the word, and then, replacing all the
$2$'s with $0$'s, gives the characteristic function
$\chi_{(X,Y)}^{0}$ of some subset $I^0$ of $[|Y|]$ with $|X|$ elements.

Given $I\in \Schub(\Gr_\omega(r,2n))$ for some $1\leq r\leq n$, we
apply this construction to $I\subset ([2n]-\bar I)$. 
One obtains $I^2\in\Part(r,2r)$ and $I^0\in\Part(r,2n-r)$.

\begin{definition}
For a partition $\lambda = (\lambda_1,\ldots,\lambda_k) \subseteq (a^b)$, \emph{i.e.}, the rectangle with $a$ columns and $b$ rows. Define $\lambda^\vee$ \emph{with respect to $(a^b)$} to be the partition $(a-\lambda_b,a-\lambda_{b-1},\ldots,a-\lambda_1)$ where we set $\lambda_i = 0$ for $i>k$. We will denote this by $\lambda^{\vee[a^b]}$.
\end{definition}

Now,
\begin{equation}
  \label{eq:dimSchub}
  \codim(\overline{BF_I})= |\tau(I^0)^{\vee[(2n-2r)^r]}| + 1/2(|\tau(I^2)^{\vee[r^r]}| + |I\cap [n]|).
\end{equation}
Moreover,
\begin{equation}
  \label{eq:dimGr}
  \dim(\Gr_\omega(r,2n))= r(2n-2r)+\frac{r(r+1)}{2}.
  \end{equation}

Let $I\in \Schub(\Gr_\omega(r,2n))$ and $A=\bar I\cap[n]$,
$A'=I\cap[n]$ be the corresponding pair in
$\Schub'(\Gr_\omega(r,2n))$. 
Set 
\begin{equation}
  \tau^0(A,A')=\tau(I^0),\qquad
\tau^2(A,A')=\tau(I^2).
\label{eq:deftauA}
\end{equation}
For later use observe that 
\begin{equation}
  \label{eq:tauI02}
  \tau(I)=\tau(I^0)+\tau(I^2).
\end{equation}

While the above discussion defines $\tau^0(A,A'), \tau^2(A,A')$ through the bijection
\eqref{eq:bijIA}, we emphasize that these partitions from Theorem~\ref{th:minlistineq} can be defined explicitly:

\begin{deflemma}
\label{lem:tau02A}
  Set $a=|A|$ and $a'=|A'|$. Write $A=\{\alpha_1<\cdots<\alpha_a\}$ and
  $A'=\{\alpha'_1<\cdots<\alpha'_{a'}\}$.
Then
$$
\begin{array}{ll}
\tau^2(A,A')_k=a+|A'\cap [\alpha_k;n]|&\forall k=1,\dots,a;\\
  \tau^2(A,A')_{l+a}=|A\cap [\alpha'_{a'+1-l}]|&\forall l=1,\dots,a';\\
\tau^0(A,A')_k=n-a-a'+|[\alpha_k;n] - (A\cup A')| &\forall k=1,\dots,a;\\
  \tau^0(A,A')_{l+a}=|[\alpha'_{a'+1-l}]- (A\cup A')|&\forall l=1,\dots,a'.
\end{array}
$$
\end{deflemma}

\begin{proof}
Write $I=A'\cup \bar A=\{i_1<\cdots<i_r\}$ with $r=a+a'$.
By definition,
\[\tau^2(A,A')_k:=\tau(I^2)_k=|\bar I\cap
[i_{a+a'+1-k}]|, \text{ \ for $1\leq k\leq a+a'$}.\] 
If $k\leq a$, $i_{a+a'+1-k}\in \bar A\subset [n+1;2n]$,  
$i_{a+a'+1-k}=\overline{\alpha_k}$ and $\bar I\cap
[i_{a+a'+1-k}]=A\cup \overline{A'\cap[\alpha_k;n]}$. The first
assertion follows.

If $k=a+l$ for some positive $l$, $i_{a+a'+1-k}\in A'\subseteq [n]$,  
$i_{a+a'+1-k}=\alpha'_{a'+1-l}$ and $\bar I\cap
[i_{a+a'+1-k}]=A\cap[\alpha'_{a'+1-l}]$.

Similarly, 
$$
\tau^0(A,A')_k:=\tau(I^0)_k=|[i_{a+a'+1-k}]\cap ([2n]-(I\cup\bar I)|, 
\text{\ for $1\leq k\leq a+a'$.}
$$ 
If $k\leq a$, $[i_{a+a'+1-k}]\cap ([2n]-(I\cup\bar I))=
([n]-(A\cup A'))\cup \overline{[\alpha_k;n]-(A\cup A')}$ (a
disjoint union).  This proves the third claim.

If $k=a+l$ with some positive $l$, $\alpha'_{a'+1-l}=i_{a+a'+1-k}\in  [n]$; last assertion follows.
\end{proof}

\subsection{The parabolic subgroup $P_0$}\label{subsec:P0}
Fix $m\geq n$.
Let $P_0$ be the subgroup of $\Sp(2m,\CC)$ of matrices
\begin{equation}
\begin{pmatrix}
  T_1&*&*\\
0&A&*\\
0&0&T_2
\end{pmatrix},\label{eq:defP0}
\end{equation}
where $T_1$ and $T_2$ are $n\times n$ upper triangular matrices and $A$ is a 
matrix in $\Sp(2m-2n,\CC)$. $P_0$ is the standard parabolic subgroup of $\Sp(2m,\CC)$ corresponding to the simple roots
$\{\alpha_{n+1},\dots,\alpha_{m}\}$.
A character $\lambda=\sum_{i=1}^m\lambda_i\varepsilon_i\in X(T)$
extends to $P_0$ if and only if $\lambda_{n+1}=\cdots=\lambda_m=0$.
Thus the set of dominant characters of $X(P_0)$ identifies with $\Par_n$. Hence
\begin{equation}
  \label{eq:7}
  \spss(m)\cap(\Par_n^\QQ)^3=\gitss(\Sp(2m,\CC),(\Sp(2m,\CC)/P_0)^3).
\end{equation}

Let $\Schub^{P_0}(\Gr_\omega(r,2m))$ be the set of  $I\in\Schub(\Gr_\omega(r,2m))$  such that the Schubert variety $\overline{BF_I}$ is $P_0$-stable. One checks (details omitted) that
\begin{equation}
\label{eqn:P0stablechar}
I\in \Schub^{P_0}(\Gr_\omega(r,2m)) \iff I\cap[n+1;2m-n]=[k;2m-n]
\text{ \ for some $k\geq m+1$.}
\end{equation}

\section{Proof of Theorems~\ref{th:cones=} and~\ref{th:minlistineq}}
\label{sec:cones=}

\begin{proposition}
\label{prop:Jul2aaa}
The inequalities~\eqref{Ineq:ABC} in Theorem~\ref{th:minlistineq} characterize
$\spss(n)$.
\end{proposition}
\begin{proof}
Since $\spss(n)=\gitss(\Sp(2n), (\Sp(2n)/B)^3$ (see Section~\ref{subsec:tensorsemi}), we may apply Theorem~\ref{th:BK}.
Let $(\lambda,\mu,\nu)\in (\Par_n)^3$. Write
$\lambda=\sum_i\lambda_i\varepsilon_i$ and similarly for $\mu$ and
$\nu$.

  Fix $1\leq r\leq n$ and $\alpha=\alpha_r\in \Delta$. Given $I\in
  \Schub(\Gr_\omega(r,2n))$, from the description of
  $\varpi_{\alpha_r^{\vee}}$ and $w_I$ it is easy to check that
  \begin{equation}
  \label{eqn:ineqconversion}
  \scal{w_I\varpi_{\alpha_r^\vee},\lambda}=\sum_{i\in I\cap[n]}\lambda_i-\sum_{i\in\bar
    I\cap[n]}\lambda_i
    \end{equation}
Then \eqref{Ineq:ABC} is obtained from \eqref{eq:BKineqGB}
associated to the triple of Schubert classes $(I,J,K)\in
\Schub(\Gr_\omega(r,2n))^3$ by setting 
\begin{align*}
    A = \bar{I}\cap[n],\ & A' = I\cap[n];\\
    B = \bar{J}\cap[n],\ & B' = J\cap[n];\\
    C = \bar{K}\cap[n],\ & C' = K\cap[n].
\end{align*}
Since the map \eqref{eq:bijIA} is bijective, it suffices to show \eqref{eq:BKcohomGB}
from Theorem~\ref{th:BK} is equivalent to 
\begin{enumerate}
\item $|A'|+|B'|+|C'|=r$, and
\item \label{cond:c1proof}
$c_{\tau^0(A,A')^{\vee[(2n-2r)^r]}, 
    \tau^0(B,B')^{\vee[(2n-2r)^r]}}^{\tau^0(C,C')}=
c_{\tau^2(A,A')^{\vee[r^r]}, \tau^2(B,B')^{\vee[r^r]}}^{\tau^2(C,C')}=1$.
\end{enumerate}

By \cite[Theorem~19]{algo}, condition~\eqref{eq:BKcohomGB} is
equivalent to 
\begin{enumerate}
\item
  $\codim(\overline{BF_I})+\codim(\overline{BF_J})+\codim(\overline{BF_K})=\dim(\Gr_\omega(r,2n))$,
  and
\item $c^{\tau(K_0)}_{\tau(I_0)^{\vee[(2n-2r)^r]}, \tau(J_0)^{\vee[(2n-2r)^r]}} = c^{\tau(K_2)}_{\tau(I_2)^{\vee[r^r]}, 
    \tau(J_2)^{\vee[r^r]}} = 1$.
\end{enumerate}
By definition, the two conditions involving
Littlewood-Richardson are the same. 
Assuming these two Littlewood-Richardson coefficients equal to one, it
remains to prove that 
$\codim(\overline{BF_I})+\codim(\overline{BF_J})+\codim(\overline{BF_K})=\dim(\Gr_\omega(r,2n))$
if and only if $|A'|+|B'|+|C'|=r$.
This directly follows from \eqref{eq:dimSchub} and \eqref{eq:dimGr}.
% see ver29 for an argument
\end{proof}

\noindent
\emph{Proof of Theorem~\ref{th:cones=}:}
By Theorem~\ref{th:KN}, the inclusion $\spss(n)\subset \spss(m)$ is
equivalent to the inclusion $\Gamma(\Sp(2n,\CC)\cap U(2n,\CC))\subset
\Gamma(\Sp(2m,\CC)\cap U(2m,\CC))$.
Here we use the symplectic form defined in Section~\ref{sec:rootC} to embed
  $\Sp(2n,\CC)$ in $\GL(2n,\CC)$.

Clearly, the following map is well-defined
$$
\begin{array}{ccl}
  \Lie(\Sp(2n,\CC)\cap U(2n,\CC))&\longto&\Lie(\Sp(2m,\CC)\cap
                                               U(2m,\CC))\\
  M=\begin{pmatrix}
    A&B\\C&D
  \end{pmatrix}
&\longmapsto&
\tilde M=\begin{pmatrix}
    A&0&B\\0&0&0\\C&0&D
  \end{pmatrix}
\end{array}
$$
where $A$, $B$, $C$ and $D$ are square matrices of size $n$, and the matrices
of these Lie algebras are described by 
\eqref{eq:sp2nu}.

Let $(\hat h_1,\hat h_2,\hat h_3)\in \Gamma(\Sp(2n,\CC) \cap U(2n,\CC))$. 
Let 
\[(M_1,M_2,M_3)\in (\Sp(2n,\CC) \cap U(2n,\CC))^3.(\sqrt{-1}\hat h_1, \sqrt{-1}\hat h_2, \sqrt{-1}\hat h_3)\]
such that $M_1+M_2+M_3=0$.

The fact that $\tilde M_1+\tilde M_2+\tilde M_3=0$ implies that
$(\hat h_1,\hat h_2,\hat h_3)\in \Gamma(\Sp(2m,\CC) \cap U(2m,\CC))$, where
$h_1,h_2,h_3$ are viewed as elements of $\Par_{m}$ is the spectrum of
$\tilde M_1$. 

To obtain the converse inclusion
\[\gitss(\Sp(2m,\CC),(\Sp(2m,\CC)/P_0)^3)=\spss(m)\cap(\Par_n^\QQ)^3\subset
\spss(n),\] 
we have to prove that any inequality~\eqref{Ineq:ABC} from
Proposition~\ref{prop:Jul2aaa}
is satisfied by the points of $\spss(m)\cap(\Par_n^\QQ)^3$; here we have
used \eqref{eq:7}. Fix such an
inequality $(A,A',B,B',C,C')$. Set 
\[I=A'\cup \bar A\subset [2n],
J=B'\cup \bar B\subset [2n], \text{\ and $K=C'\cup \bar C\subset [2n]$.}\]
Similarly for $m$, set 
\[\tilde I=A'\cup\{2m+1-i\,:\,i\in A\}, 
\tilde J=B'\cup\{2m+1-i\,:\,i\in B\}\] 
and 
\[\tilde K=C'\cup\{2m+1-i\,:\,i\in C\};\] these
are subsets of $[2m]$.
Set also $a'=|A'|$, $b'=|B'|$ and $c=|C|$.

Notice that $(\tilde I^2)^0,(\tilde J^2)^0,(\tilde K^2)^0\subseteq 0^r = \emptyset$ . Thus, trivially,
\begin{equation}
\label{eqn:Jul1aaa}
c_{\tau((\tilde I^2)^0)^\vee, \tau((\tilde J^2)^0)^\vee}^{\tau((\tilde K^2)^0)}=c_{\emptyset,\emptyset}^{\emptyset}=1.
\end{equation}
Also, $(\tilde I^2)^2=\tilde I^2, (\tilde J^2)^2=\tilde J^2, (\tilde K^2)^2=\tilde K^2$.
Since $\tau(\tilde I^2)=\tau(I^2):=\tau^2(A,A')$, $\tau(\tilde J^2)=\tau(J^2):=\tau^2(B,B')$ and
$\tau(\tilde K^2)=\tau(K^2):=\tau^2(C,C')$, we have
\begin{equation}
\label{eqn:Jul1bbb}
c_{\tau((\tilde I^2)^2)^{\vee[r^r]}, \tau((\tilde J^2)^2)^{\vee [r^r]}}^{\tau((\tilde K^2)^2)}=
c_{\tau(\tilde I^2)^{\vee[r^r]},\tau(\tilde J^2)^{\vee [r^r]}}^{\tau(\tilde K^2)}=
c_{\tau^2(A,A')^{\vee[r^r]}\; \tau^2(B,B')^{\vee[r^r]}}^{\tau^2(C,C')}=1.
\end{equation}
We apply \cite[Theorem~8.2]{algo} to $\Gr_\omega(r,2r)$ and the triple
$\tilde I^2, \tilde J^2, \tilde K^2$. The equations (\ref{eqn:Jul1aaa}) and 
(\ref{eqn:Jul1bbb}) mean that condition (iii) of said theorem holds. Hence by 
part (ii) of \emph{ibid.},
\begin{equation}
[\overline{BF_{\tilde I^2}}]\cdot
[\overline{BF_{\tilde J^2}}]\cdot
[\overline{BF_{\tilde K^2}}]=[pt]
\in \Ho^*(\Gr_\omega(r,2r),\ZZ).
\label{eq:condLG}
\end{equation}

One can easily check that 
$$
\begin{array}{l}
  \tau(\tilde K^0)=[2(m-n)]^{c}+\tau(K^0),\\
\tau(\tilde
  I^0)^{\vee[(2m-2r)^r]}=[2(m-n)]^{a'}+\tau(I^0)^{\vee[(2n-2r)^r]},\\
\tau(\tilde J^0)^{\vee[(2m-2r)^r]}=[2(m-n)]^{b'}+\tau(J^0)^{\vee[(2n-2r)^r]}.
\end{array}
$$
The assumption $a'+b'=c$ and the semigroup property of $\lrs$ implies
that 
\begin{equation}
\label{eqn:Jul1ccc}
c_{\tau(\tilde
  I^0)^{\vee[(2m-2r)^r]}, \tau(\tilde
  J^0)^{\vee[(2m-2r)^r]}}^{\tau(\tilde K^0)}\neq 0.
  \end{equation}
Next we apply \cite[Proposition~8.1]{algo} to the $\tilde I, \tilde J, \tilde K$ and
the space $\Gr_\omega(r,2m)$; equations (\ref{eq:condLG}) and (\ref{eqn:Jul1ccc})
mean  that condition (iii) holds. Hence by (i) of \emph{ibid.} and (\ref{eqn:P0stablechar}), 
\begin{equation}
[\overline{P_0F_{\tilde I}}]\odot_0
[\overline{P_0F_{\tilde J}}]\odot_0
[\overline{P_0F_{\tilde K}}]=d[pt]
\in \Ho^*(\Gr_\omega(r,2m),\ZZ),
\label{eq:cohom}
\end{equation}
for some nonzero $d$. Now use
Proposition~\ref{prop:coneineqplus}, which 
 shows that \eqref{Ineq:ABC} is a case of (\ref{eq:domineq}) which holds
 on $\gitss(\Sp(2m,\CC),(\Sp(2m,\CC)/P_0)^3)=\spss(m)\cap(\Par_n^\QQ)^3$,
 as desired.\qed

\medskip
\noindent
\emph{Proof of Theorem~\ref{th:minlistineq}:}
This follows from Theorem~\ref{th:cones=} and Proposition~\ref{prop:Jul2aaa}. \qed

\begin{example}\label{exa:running}
Let $n=4, r=3$. Let
$$
A=B'=C'=\emptyset,\ 
A'=\{2,3,4\}, B=\{1,2,4\}, C=\{1,3,4\},
$$
giving a triple $((A,A'),(B,B'),(C,C'))$ in
$(\Schub'(\Gr_\omega(3,8)))^3$ satisfying conditions (1) and (2) from
Theorem~\ref{th:minlistineq}. The corresponding triple in $\Schub(\Gr_\omega(3,8))$ is
\[I=\{2, 3, 4\}, J=\{5, 7, 8\}, K=\{5, 6, 8\} \subseteq [8].\]
Thus 
\[\tau(I)=(1,1,1), \tau(J)=(5,5,4), \tau(K)=(5,4,4)\subseteq (5^3);\]
The three associated characteristic functions
$\chi_{I\subset[2n]-\bar I}, \chi_{J\subset[2n]-\bar J}, \chi_{K\subset[2n]-\bar K}$ respectively are
$$
21110002,\quad 00201211,\quad 02001121.
$$
Thus the characteristic functions are
$$
\begin{array}{lll}
 \chi_{I^2}=111000,& \chi_{J^2}=000111,& \chi_{K^2}=000111;\\
\chi_{I^0}=01110,& \chi_{J^0}=01011,& \chi_{K^0}=01101.
\end{array}
$$
Now,
$$
\begin{array}{llll}
  I^2=\{1,2,3\}&\tau(I^2)=000&J^2=K^2=\{4,5,6\}&\tau(J^2)=\tau(K^2)=333\\
I^0=\{2,3,4\}&\tau(I^0)=111\\
J^0=\{2,4,5\}&\tau(J^0)=221&K^0=\{2,3,5\}&\tau(K^0)=211
\end{array}
$$
The reader can check that
$$
\begin{array}{l}
c_{\tau^0(A,A')^{\vee[(2n-2r)^r]},
    \tau^0(B,B')^{\vee[(2n-2r)^r]}}^{\tau^0(C,C')}=c^{\tau(K_0)}_{\tau(I_0)^{\vee[(2n-2r)^r]},
    \tau(J_0)^{\vee[(2n-2r)^r]}} =c_{(1,1,1),(1)}^{(2,1,1)}=1,\\  
c_{\tau^2(A,A')^{\vee[r^r]}, \tau^2(B,B')^{\vee[r^r]}}^{\tau^2(C,C')}=c^{\tau(K_2)}_{\tau(I_2)^{\vee[r^r]}, \tau(J_2)^{\vee[r^r]}} = c_{(3,3,3),(0,0,0)}^{(3,3,3)} =1.
\end{array}
$$
Hence, by Theorem~\ref{th:minlistineq}, 
$-\lambda_2-\lambda_3-\lambda_4+\mu_1+\mu_2+\mu_4+\nu_1+\nu_3+\nu_4\geq 0$
is one of the inequalities defining $\spss(4)$.\qed
\end{example}

\section{The truncated tensor cone}\label{sec:truncated}

In this section, we characterize the \emph{truncated tensor cone} of
$\spss(m)$, that is, $\spss(m)\cap(\Par_n^\QQ)^3$ where $m>n$.
By (\ref{eqn:firstobs}), this implies another set of inequalities
for $\nlss(n)$.
 
We first need the following result, a generalization of Theorem~\ref{th:BK}:

\begin{theorem}\label{th:BKP}
Here $X=G/P\times G/Q\times G/R$.
   Let $(\lambda,\mu,\nu)\in X(P)\times X(Q)\times X(R)$ be dominant. 
Then $\Li_{(\lambda,\mu,\nu)}\in \gitss(G,X)$ if and only if for any
simple root $\alpha$, for any 
$$
(u,v,w)\in {W_P}\backslash W/ W_{P^\alpha}\times {W_Q}\backslash W/W_{P^\alpha}\times {W_R}\backslash W/W_{P^\alpha}
$$ such that
\begin{equation}
  \label{eq:BKcohomGP}
  [\overline{PuP^\alpha/P^\alpha}]\odot_0
[\overline{QvP^\alpha/P^\alpha}]\odot_0
[\overline{RwP^\alpha/P^\alpha}]=[pt]
\in \Ho^*(G/P^\alpha,\ZZ),
\end{equation}
\begin{equation}
  \label{eq:BKineqGP}
  \scal{u\varpi_{\alpha^\vee},\lambda}+\scal{v\varpi_{\alpha^\vee},\mu}+\scal{w\varpi_{\alpha^\vee},\nu}\leq 0.
\end{equation}
\end{theorem}

\begin{proof}
$\gitss(G,X)$ is characterized by
Proposition~\ref{prop:cone1}; let $(C,\tau)$ and a choice of 
\[(u',v',w')\in {W_P}\backslash W/ W_{P(\tau)}\times {W_Q}\backslash W/W_{P(\tau)}\times {W_R}\backslash W/W_{P(\tau)}\]
be as in that proposition.
Since every inequality \eqref{eq:BKineqGP} appears in Proposition~\ref{prop:cone1} with $\tau = \varpi_{\alpha^\vee}$,
it suffices to show that (\ref{eq:wcineq}) is implied by the inequalities in Theorem~\ref{th:BKP}.

Write
$$
\tau=\sum_{\alpha\in\Delta}n_\alpha\varpi_{\alpha^\vee},
$$
where $\Delta$ is the set of simple roots.
Since $\tau$ is dominant the $n_\alpha$'s are nonnegative. 
Set
$$
\Supp(\tau):=\{\alpha\in\Delta\,:\,n_\alpha\neq 0\}.
$$ 

Fix any $\alpha\in \Supp(\tau)$, $P^\alpha:=P(\varpi_{\alpha^\vee})$ contains $P(\tau)$. 
Let 
\[\pi\,:\, G/P(\tau)\longto G/P^\alpha\] 
denote the associated projection. 
By \cite[Theorem~1.1 and Section~1.1]{Rich:mult} (see also \cite{multi}), condition~\eqref{eq:defwc} implies there are 
\[(u,v,w)\in {W_P}\backslash W/ W_{P^\alpha}\times {W_Q}\backslash W/W_{P^\alpha}\times {W_R}\backslash W/W_{P^\alpha},\]
such that condition~(\ref{eq:BKcohomGP}) holds and such that $(u,v,w)$ and $(u',v',w')$ define the same cosets in ${W_P}\backslash W/ W_{P^\alpha}\times {W_Q}\backslash W/W_{P^\alpha}\times {W_R}\backslash W/W_{P^\alpha}$. 
Therefore, inequality \eqref{eq:BKineqGP} is the same as
\[\scal{u'\varpi_{\alpha^\vee},\lambda}+\scal{v'\varpi_{\alpha^\vee},\mu}+\scal{w'\varpi_{\alpha^\vee},\nu}\leq 0.\]
Therefore each of the inequalities \eqref{eq:wcineq} can be written as a linear combination of \eqref{eq:BKineqGP}.
Hence the inequalities of the theorem imply and are implied by the
inequalities of Proposition~\ref{prop:cone1}, so the result follows.
\end{proof}

We now deduce from Theorem~\ref{th:BKP} the following statement.

\begin{proposition}
  \label{prop:ineqP0}
Let $(\lambda,\mu,\nu)$ in $\Par_n$ and $m\geq n$. 
Then $(\lambda,\mu,\nu)\in \spss(m)$ if and only
if
\begin{equation}
  \label{eqn:BKSp4nP0}
|\lambda_{I\cap [n]}|-
|\lambda_{\bar I \cap [n]}|+|\mu_{J\cap [n]}|
-|\mu_{\bar  J\cap [n]}|+|\nu_{K\cap [n]}|
-|\nu_{\bar K\cap[n]}|\leq 0,
\end{equation}
for any $1\leq r\leq m$ and $(I,J,K)\in \Schub^{P_0}(\Gr_\omega(r,2m))^3$ such that
\begin{enumerate}
\item \label{cond:sumrP0}
$|I\cap[m]|+|J\cap[m]|+|K\cap[m]| = r$, and
  \item
   $ c_{\tau(I_0)^{\vee[(2m-2r)^r]},\tau(J_0)^{\vee[(2m-2r)^r]}}^{\tau(K_0)}=c_{\tau(I_2)^{\vee[r^r]},\tau(J_2)^{\vee [r^r]}}^{\tau(K_2)}=1.$
\end{enumerate}
\end{proposition}

\begin{proof}
  We already observed \eqref{eqn:ineqconversion} that inequality~\eqref{eqn:BKSp4nP0} is
  inequality~\eqref{eq:wcineq} in our context.  Regarding
  Theorem~\ref{th:BKP}, the only thing to prove is that
  condition~\eqref{eq:defwc} associated to $(I,J,K)$ is equivalent to
  the two conditions of the proposition. This is \cite[Theorem~8.2]{algo}.
\end{proof}

\emph{A priori}, Proposition~\ref{prop:ineqP0} could contain redundant
inequalities. In view of Theorem~\ref{th:minlistineq},
an affirmative answer to this question would imply irredundancy:

\begin{question}
  Does any $(I,J,K)\in \Schub^{P_0}(\Gr_\omega(r,2m))^3$ occurring in
  Proposition~\ref{prop:ineqP0} satisfy
  \begin{enumerate}
  \item $I\cap[n+1,2m-n]=J\cap[n+1,2m-n]=K\cap[n+1,2m-n]=\emptyset$;
\item $c_{\tau(\hat I_0)^{\vee[(2n-2r)^r]},\tau(\hat
    J_0)^{\vee[(2n-2r)^r]}}^{\tau(\hat K_0)}=1$,
  \end{enumerate}
where $\hat I=I\cap[n]\cup\{i-2(m-n)\,:\,i\in\bar I\cap [m+1,2m]\}$,
and, $\hat J$ and $\hat K$ are defined similarly.
\end{question}

\begin{proof}[Proof of Theorem~\ref{th:reduction}.]
Fix an inequality $(A,A',B,B',C,C')$ from \eqref{Ineq:ABC}. It is irredundant
for the full-dimensional cone
\[\nlss(n)=\spss(2n)\cap(\Par_n)^3\subset {\mathbb R}^{3n}.\] 
Thus, it has to appear in
Proposition~\ref{prop:ineqP0} for $m=2n$. 
Let $(\tilde I,\tilde J,\tilde K)\in \Schub^{P_0}(\Gr_\omega(\tilde
r,4n))^3$ be the associated Schubert triple.
Set $\tilde A'=\tilde I\cap[2n]$, $\tilde A =\overline{\tilde
  I}\cap[2n]$, \emph{etc}.  
  Since $(\tilde I,\tilde J,\tilde K)\in \Schub^{P_0}(\Gr_\omega(\tilde
r,4n))^3$, $\tilde A', \tilde B', \tilde C'\subset [n]$ (by \eqref{eqn:P0stablechar}). 
Thus, comparing \eqref{Ineq:ABC} and \eqref{eqn:BKSp4nP0}, we have
$$
\begin{array}{lll}
A=\tilde A\cap[n]&B=\tilde B\cap[n]&C=\tilde C\cap[n]\\
A'=\tilde A'\cap[n]=\tilde A'&B'=\tilde B'\cap[n]=\tilde B'&C'=\tilde C'\cap[n]=\tilde C'.
\end{array}
$$

Now, Proposition~\ref{prop:ineqP0}\eqref{cond:sumrP0} and Theorem~\ref{th:minlistineq}\eqref{cond:sumr}
imply that $r=\tilde r$. 
In particular, $|\tilde A|+|\tilde A'|=|A|+|A'|=r$ and $A=\tilde
A$. Similarly, $B=\tilde
B$ and $C=\tilde C$.

Let $\alpha$ be the simple root of $\Sp(4n,\CC)$ associated to
$r$. Observe that the Levi subgroup of $P^\alpha$ has type
$A_{r-1}\times C_{2n-r}$.
Let $u,v,w\in W^{P^\alpha}$ corresponding to $(\tilde
A',\tilde A), (\tilde B',\tilde B)$ and $(\tilde C',\tilde C)$, respectively.
Proposition~\ref{prop:ineqP0} and its proof show that
\eqref{eq:BKcohomGP} holds with $P=Q=R=P_0$.
In particular, one can apply the reduction rule proved in
\cite[Theorem~3.1]{Roth:red} or \cite[Theorem~1]{reduction}: $\mult_{\lambda,\mu,\nu}^{2n}$
is a tensor multiplicity for the Levi subgroup of $P^\alpha$ of type
$A_{r-1}\times C_{2n-r}$.
The factor $c_{\lambda_{A,A'},\mu_{B,B'}}^{\nu_{C,C'}^*}$ in the
theorem corresponds to the factor of type $A_{r-1}$. Adding zeros,
consider $\lambda$ as an element of $\Par_{2n}$.
Then the dominant weights to consider for the factor $C_{2n-r}$ are
$\lambda_{[2n]-(\tilde A\cup\tilde A')}, \mu_{[2n]-(\tilde
    B\cup\tilde B')}, \nu_{[2n]-(\tilde
    C\cup\tilde C')}$. 
Since these partitions have length at most $n-r$, the tensor
multiplicity for the factor of $C_{2n-r}$ is a Newell-Littlewood
coefficient. The theorem follows.
\end{proof}

\section{Application to Conjecture~\ref{conj:saturation}}
\label{sec:apptosat5}
 \begin{corollary}[of Theorem~\ref{th:minlistineq}]\label{cor:saturation}
 Conjecture~\ref{conj:saturation} holds for $n\leq 5$.
 \end{corollary}
  The proof is
 computational and uses the software {\sf Normaliz} \cite{normaliz}.

Fix $n\geq 2$ and consider the cone $\spss(n)$. Consider the two
lattices $\Lambda=\ZZ^{3n}$ and 
$$
\Lambda_2=\{(\lambda,\mu,\nu)\in(\ZZ^n)^3 :
|\lambda|+|\mu|+|\nu|\mbox{ is even}\}.
$$
Then $\nls(n)\subset\Lambda_2\cap \spss(n)$. 
Conjecture~\ref{conj:saturation} asserts
that the converse inclusion holds. 
The set $\Lambda_2\cap \spss(n)$ is a semigroup of $\Lambda_2$ defined
by a family of linear inequalities (explicitly given by
Theorem~\ref{th:minlistineq}). Using {\sf Normaliz} \cite{normaliz} one can compute (for
small $n$) the minimal set of generators, \emph{i.e.}, the \emph{Hilbert basis}, 
for this semigroup. Hence, to prove
Corollary~\ref{cor:saturation} one can proceed as follows:
\begin{enumerate}
\item Compute the list of inequalities given by
  Theorem~\ref{th:minlistineq}.
\item Compute the Hilbert basis of  $\Lambda_2\cap \spss(n)$ using
  {\sf Normaliz}.
\item Check $N_{\lambda,\mu,\nu}>0$ for any $(\lambda,\mu,\nu)$ in the
  Hilbert basis.
\end{enumerate}

The table below summarizes our computations; see \cite{programmation}.

\begin{center}
  \begin{tabular}{|c|c|c|c|c|c|}
    \hline
    $n$&$\#$ facets&\# EHI&$\#$ rays&$\#$ Hilb $\Lambda_2\cap\spss$&$\#$
                                                            Hilb $\Lambda\cap\spss$\\
    \hline
    \hline
    2& 6+18&18&12&13&20\\
    \hline
    3&9+93 &100&51&58&93\\
    \hline
    4&12+474 &662&237&302&451\\
    \hline
    5&15+2\,421&5\,731&1\,122&1\,598& 2\,171\\
    \hline
  \end{tabular}
\end{center}

In the column ``$\#$ facets'' there are the number of partition
inequalities (like $\lambda_1\geq\lambda_2$) plus the number of
inequalities~\eqref{Ineq:ABC} given by Theorem~\ref{th:minlistineq}. 
The next column counts the inequalities~\eqref{Ineq:EH} given by applying
Theorem~\ref{th:GOYconj}. 
The number of extremal rays of the cone $\spss(n)$ is also given. 
The two last column are the cardinalities of the Hilbert bases of the
two semigroups  $\Lambda_2\cap\spss(n)$ and $\Lambda\cap\spss(n)$.

\section{Littlewood-Richardson coefficients}
\label{sec:LRprelim}

We recall material \cite{FultonYT, Fulton.Harris} on
Littlewood-Richardson coefficients and their role in representation theory of the general linear group. 

\subsection{Representations of $\GL(n,\CC)$}

The irreducible rational representations $V(\lambda)$ of $\GL(n,\CC)$ are
indexed by their highest weight
\[\lambda\in\Lambda_n^+=\{(\lambda_1\geq\cdots\geq\lambda_n)\,:\,\lambda_i\in\ZZ\}\supset\Par_n.\] 
One has tensor product multiplicities $c_{\lambda,\mu}^{\nu}$ defined 
for any $\lambda,\mu,\nu \in\Lambda_n^+$ by:
\begin{equation}\label{eqn:Jun3aaa}
V(\lambda)\otimes V(\mu)=\bigoplus_{\nu\in \Lambda_n^+} V(\nu)^{\oplus c_{\lambda,\mu}^\nu}.
\end{equation}
When $\lambda,\mu,\nu\in\Par_n$, $c_{\lambda,\mu}^\nu$ is
the Littlewood-Richardson coefficient (which is why we use the same notation).

The dual representation $V(\lambda)^*$ has highest weight
\[\lambda^*=(-\lambda_n\geq\cdots\geq-\lambda_1)\in\Lambda_n^+.\] 
Moreover, for any $a\in\ZZ$, 
\begin{equation}\label{eqn:Jun3bbb}
V(\lambda+a^n)=(\det)^a\otimes V(\lambda).
\end{equation}
Consequently, for any $\lambda,\mu,\nu\in\Lambda^+_n$,
\begin{equation}
  \label{eq:LRinvtrans}
  c_{\lambda,\mu}^\nu=c_{\lambda+a^n,\mu+b^n}^{\nu+(a+b)^n}=c_{\lambda^*+a^n,\mu^*+b^n}^{\nu^*+(a+b)^n};
\end{equation}
this is \cite[Theorem~4]{BOR}.
For $a$ and $b$ big enough, Formula~\eqref{eq:LRinvtrans} implies that
$c_{\lambda,\mu}^\nu$ is a Littlewood-Richardson coefficient.

Let $\nu^t$ denote the conjugate of $\nu$.
Since
$c^{\nu}_{\lambda,\mu} = c^{\nu^t}_{\lambda^t,\mu^t}$,
by \eqref{eq:LRinvtrans}, 
\begin{equation}
\label{eq:Jun12bbb}
    c^{\nu}_{\lambda,\mu} = c^{\nu^t}_{\lambda^t,\mu^t} = c^{(\nu^t)^{\vee[(n+m)^{a+b}]}}_{(\lambda^t)^{\vee[n^{a+b}]},(\mu^t)^{\vee[m^{a+b}]}} = c^{(\nu^{\vee [(a+b)^{n+m}]})^t}_{(\lambda^{\vee[(a+b)^n]})^t,(\mu^{\vee[(a+b)^m]})^t} = c^{\nu^{\vee[(a+b)^{n+m}]}}_{\lambda^{\vee[(a+b)^n]},\mu^{\vee[(a+b)^m]}},
\end{equation}
for any $m\geq \ell(\mu)$.

\subsection{Six-fold NL-coefficients}\label{sec:GLstable}
Let $p$, $q$ and $m$ be positive integers such that $p+q\leq m$. 
Following R.~Howe-E-C.~Tan-J.~Willenbring \cite{HTWbranching}, to any
$\lambda^+\in\Par_p$ and $\lambda^-\in\Par_q$, we associate the
following element in $\Lambda_m^+$:
\[
[\lambda^+,\lambda^-]_m =  (\lambda^+_1,\lambda^+_2,\ldots,
\lambda^+_p,\underbrace{0,\ldots,0}_{m-p-q},-\lambda^-_{q},\ldots,-\lambda^-_1).\]

Let $V(\lambda)\boxtimes V(\mu)$ be the irreducible representation of $\GL(n,\CC)\times \GL(n,\CC)$,
where $\boxtimes$ refers to external tensor product. View $\GL(n,\CC)\subset \GL(n,\CC)\times \GL(n,\CC)$
under the diagonal embedding. The associated \emph{branching coefficient} is
\[[V(\nu):  V(\lambda)\boxtimes V(\mu)]:= \dim \Hom_{\GL(n,\CC)}(V(\nu), V(\lambda)\boxtimes
V(\mu)|_{\GL(n,\CC)}).\]

\begin{proposition}\cite[Section 2.1.1]{HTWbranching},\cite{King71}\label{thm:GLbranching}
Let $\lambda^\pm$, $\mu^\pm$ and $\nu^\pm$ be six partitions.
Let $p,q,r$ and $s$ be four nonnegative integers such that 
$$
\begin{array}{lll}
  \ell(\lambda^+)\leq p&\ell(\mu^+)\leq r&\ell(\nu^+)\leq p+r\\
\ell(\lambda^-)\leq q&\ell(\mu^-)\leq s&\ell(\nu^-)\leq q+s\\
\end{array}
$$
Let $m$ be a positive integer such that $m\geq p+q+r+s$. Then 
\begin{align*}
N_{\mu^+,\nu^+,\lambda^+,\mu^-,\nu^-,\lambda^-}=  & \ \dim \Hom_{\GL(n,\CC)}(V(\nu), V(\lambda)\boxtimes
V(\mu))\\
= & \ c_{[\lambda^+,\lambda^-]_m, [\mu^+,\mu^-]_m}^{[\nu^+,\nu^-]_m}.
\end{align*}
\end{proposition}
\begin{proof}
The first equality is in \cite[Section 2.1.1]{HTWbranching} who credit \cite{King71}. The second statement follows from
\cite[p.~427]{Fulton.Harris}.\footnote{Using (\ref{eqn:Jun3bbb}) one can equate this with a Littlewood-Richardson coefficient $c_{\tilde\lambda,\tilde\mu}^{\tilde\nu}$ where $\tilde\lambda,\tilde\mu,\tilde\nu \in \Par_m$.}
\end{proof}

Conversely, any Littlewood-Richardson coefficient is a six-fold
NL-coefficient. More precisely,
$c_{\lambda,\mu}^\nu=N_{\mu,\nu,\lambda,\emptyset,\emptyset,\emptyset}$,
which corresponds to the case when $\lambda^-=\mu^-=\nu^-=\emptyset$
in Proposition~\ref{thm:GLbranching}.

We now use Proposition~\ref{thm:GLbranching} to rephrase
Theorem~\ref{th:GOYconj}. Fix $A$, $A'$, $B$, $B'$, $C$ and $C'$
subseets of $[n]$
satisfying the two first conditions of Theorem~\ref{th:GOYconj}. 
Set $p=|B'|$, $r=|C'|$, $q=n-p$ and $s=n-q$. 
 Observe that $|A|=p+r$, $|B|\leq q$, $|C|\leq s$ and $|A'|\leq
 n-p-r\leq p+q$. Set finally $m=2n=p+q+r+s$. 
Since $\ell(\tau(A))\leq |A|$, Proposition~\ref{thm:GLbranching}  implies that
$$
N_{\tau(A), \tau(C'), \tau(B), \tau(A'), \tau(C), \tau(B')}=c_{[\tau(B'),\tau(B)]_m, [\tau(C'),\tau(C)]_m}^{[\tau(A),\tau(A')]_m}
$$
is a Littlewood-Richardson coefficient for $\GL_{2n}(\CC)$.
In particular, in Theorem~\ref{th:GOYconj}, condition $(3)$ can be
replaced by 
$$
(3')\quad c_{[\tau(B'),\tau(B)]_m, [\tau(C'),\tau(C)]_m}^{[\tau(A),\tau(A')]_m}>0 .
$$

We now observe that Proposition~\ref{thm:GLbranching} and Knutson-Tao
saturation \cite{KT:saturation} implies the saturation result for the
six-fold NL-coefficients from the introduction (Proposition~\ref{prop:satNL6}).

\section{Extended Horn inequalities and the proof of Theorem~\ref{th:GOYconj}}\label{sec:proofGOYconj}

\subsection{Extended Horn inequalities}
\label{sec:relationtoGOY}

We recall the following notion from \cite{GOY20b}:
\begin{definition}
\label{def:main}
An \emph{extended Horn inequality} on $\Par_n^3$ is
\begin{equation}
\label{Ineq:Grand}
0\leq |\lambda_A|-|\lambda_{A'}|+|\mu_B|-|\mu_{B'}|+|\nu_C|-|\nu_{C'}|
\end{equation}  
\end{definition}
\noindent
where $A,A',B,B',C,C' \subseteq [n]$ satisfy
\begin{itemize}
\item[(I)]
$A \cap A'= B \cap B' = C \cap C' = \emptyset$
\item[(II)]
$|A| = |B'|+|C'|,|B| =|A'|+|C'|, |C| = |A'|+|B'|$
\item[(III)]
There exists $A_1,A_2,B_1,B_2,C_1,C_2\subseteq [n]$ such that:
\begin{itemize}
\item[(i)]
$|A_1|=|A_2|=|A'|, |B_1|=|B_2|=|B'|, |C_1|=|C_2|=|C'|$
\item[(ii)]
$c^{\tau(A')}_{\tau(A_1),\tau(A_2)}, c^{\tau(B')}_{\tau(B_1),\tau(B_2)}, c^{\tau(C')}_{\tau(C_1),\tau(C_2)}>0$
\item[(iii)]
$c^{\tau(A)}_{\tau(B_1),\tau(C_2)}, c^{\tau(B)}_{\tau(C_1),\tau(A_2)}, c^{\tau(C)}_{\tau(A_1),\tau(B_2)}>0.$
\end{itemize}
\end{itemize}

\begin{definition}
The \emph{extended Horn cone} is:
\begin{equation}
  \label{eq:3}
  \EH(n):=\{(\lambda,\mu,\nu)\in (\Par_n^{\mathbb Q})^3\,:\, \mbox{inequalities
    \eqref{Ineq:Grand} are satisfied}\}.
\end{equation}
\end{definition}

Let 
\[\overline{\EH}(n)=\EH(n)\cap \{(\lambda,\mu,\nu)\in (\Par_n)^3: 
|\lambda|+|\mu|+|\nu|\mbox{ is even}\}.\]

\begin{conjecture}[{\cite[Conjecture~1.4]{GOY20b}}]
\label{conj:theGOYthing}
If $(\lambda,\mu,\nu)\in {\overline \EH}(n)$ then $N_{\lambda,\mu,\nu}>0$. 
\end{conjecture}

We will prove a weakened version of Conjecture~\ref{conj:theGOYthing}:

\begin{theorem}[\emph{cf.} {\cite[Conjecture~1.4]{GOY20b}}]
\label{theGOYproof}
$\EH(n)=\nlss(n)$.
\end{theorem} 

Consequently, we are able to answer an issue raised in \cite[Section~1]{GOY20b}:

\begin{corollary}
\label{cor:satimplies}
Conjecture~\ref{conj:saturation} implies
Conjecture~\ref{conj:theGOYthing}.
\end{corollary}

Corollary~\ref{cor:satimplies} is analogous to the situation 
in Zelevinsky's \cite{Zelev}, before \cite{KT:saturation}.

The following shows that Theorem~\ref{th:GOYconj} is
equivalent to Theorem~\ref{theGOYproof}.

\begin{lemma}
\label{lemma:sextuple}
A sextuple $(A,A',B,B',C,C')$ of subsets of $[n]$ parametrizes an
extended Horn inequality if and only if it appears in Theorem~\ref{th:GOYconj}.
\end{lemma}
\begin{proof}
Definition~\ref{def:main}  implies
\[c_{\tau(B_1),\tau(C_2)}^{\tau(A)} c_{\tau(C_2), \tau(C_1)}^{\tau(C')} c_{\tau(C_1), \tau(A_2)}^{\tau(B)} c_{\tau(A_2), \tau(A_1)}^{\tau(A')} c_{\tau(A_1), \tau(B_2)}^{\tau(C)}
c_{\tau(B_2), \tau(B_1)}^{\tau(B')}>0.\]
Since $\tau(A_1),\,\tau(A_2)\dots$ have length at most $n$, this
implies 
$N_{\tau(A),\tau(C'),\tau(B),\tau(A'),\tau(C),\tau(B')}\neq 0$, and thus 
$(\tau(A), \tau(C'), \tau(B), \tau(A'), \tau(C), \tau(B'))\in \operatorname{NL}^6\!\operatorname{-sat}(r)$.

Conversely, if $(\tau(A), \tau(C'), \tau(B), \tau(A'), \tau(C), \tau(B'))\in \operatorname{NL}^6\!\operatorname{-sat}(r)$, by Proposition~\ref{prop:satNL6}, $N_{\tau(A),\tau(C'),\tau(B),\tau(A'),\tau(C),\tau(B')}\neq 0$.
Therefore there exists $\alpha_1,\alpha_2,\ldots,\alpha_6\in \Par_n$ such that
 \[c_{\alpha_1,\alpha_2}^{\tau(A)}c_{\alpha_2,\alpha_3}^{\tau(C')}c_{\alpha_3,\alpha_4}^{\tau(B)}
c_{\alpha_4,\alpha_5}^{\tau(A')}c_{\alpha_5,\alpha_6}^{\tau(C)}c_{\alpha_6,\alpha_1}^{\tau(B')}>0.\]
Set $a=|A|$. Then, the Young diagram of  $\tau(A)$ is contained in
the rectangle $a\times(n-a)$. 
But the nonvanishing of $c_{\alpha_1,\alpha_2}^{\tau(A)}$ implies
that $\alpha_1\subset\tau(A)$. Hence there exists $B_1\subseteq [n]$
such that $\tau(B_1)=\alpha_1$. Similarly, we can pick $C_2,C_1,A_2,A_1,B_2\subseteq [n]$ such that  
$\tau(C_2)=\alpha_2, \tau(C_1)=\alpha_3$ \emph{etc.} that satisfy Definition~\ref{def:main}.
\end{proof}

\subsection{Proof of Theorem~\ref{th:GOYconj}}
($\Rightarrow$) By Lemma~\ref{lemma:sextuple}, 
${\EH}(n)$ is the cone defined by the inequalities in
Theorem~\ref{th:GOYconj}. Now, $\nlss(n)\subseteq {\EH}(n)$ is immediate from \cite[Theorem~1]{GOY20b}.

($\Leftarrow$) Fix an inequality~\eqref{Ineq:ABC}
associated to $(A,A',B,B',C,C')$ appearing in
Theorem~\ref{th:minlistineq}. We now show the even stronger statement that  
\begin{equation}
\label{eqn:Nnotzero}
%N_{\tau(A'), \tau(B), \tau(C'), \tau(A), \tau(B'), \tau(C)}\neq 0.
N_{\tau(A),\tau(C'),\tau(B),\tau(A'),\tau(C),\tau(B')}\neq 0.
\end{equation}
This would imply that the
inequality appears in Theorem~\ref{th:GOYconj}, completing the proof.

Set $a=|A|, a'=|A'|, b=|B|, b'=|B'|, c=|C|, c'=|C'|$ and $r=|A|+|A'|$.
Let $I\in\Schub(\Gr_\omega(r,2n))$ be associated
to
$(A,A')\in\Schub'(\Gr_\omega(r,2n))$, under (\ref{eq:bijIA}).
Similarly define $J$ and $K$.
By \eqref{eq:tauI02} and the condition \eqref{cond:c1} in
Theorem~\ref{th:minlistineq}, the semigroup property of nonzero
LR-coefficients implies 
\begin{equation}
  \label{eq:81}
  c^{\tau(K)}_{\tau(I)^{\vee[(2n-r)^r]},\tau(J)^{\vee[(2n-r)^r]}}\neq 0.
\end{equation}

Fix a nonnegative integer $k$. Note that $c=a'+b'$ by condition \eqref{cond:sumr} in
Theorem~\ref{th:minlistineq}, and
$c_{(a')^k,(b')^k}^{(c^k)}=1$. 
Using, once more, the semigroup property one gets
\begin{equation}
  \label{eq:82}
  c^{\tau(K)+(c^k)}_{\tau(I)^{\vee[(2n-r)^r]+((a')^k)},\tau(J)^{\vee[(2n-r)^r]}+((b')^k)}> 0.
\end{equation}

Observing Figure~\ref{fig:tauIA},  
\begin{equation}
  \label{eq:83}
  	(\tau(I)^\vee)^t = [\tau(A)^t, \tau(A')^t]_{2n-r}+((a')^{2n-r}).
\end{equation}
Similarly,
\begin{equation}
  \label{eq:84}
  	(\tau(I)^\vee+(a')^k)^t = [\tau(A)^t, \tau(A')^t]_{m}+((a')^{m}),
\end{equation}
and
\begin{equation}
(\tau(K)+(c^k))^t = [\tau(C')^t, \tau(C)^t]_{m}+(c^{m}),
\label{eq:85}
\end{equation}
where $m = 2n-r+k$.  
Now, by \eqref{eq:LRinvtrans}, \eqref{eq:Jun12bbb}, \eqref{eq:84} and
\eqref{eq:85}, conditions~\eqref{eq:82} implies
\begin{equation}
  \label{eq:87}
  c^{[\tau(C')^t, \tau(C)^t]_{m}}_{[\tau(A)^t, \tau(A')^t]_m, [\tau(B)^t, \tau(B')^t]_m}>0.
\end{equation}

On the other hand, for $k$ (and hence $m$) big enough, we can apply
Proposition~\ref{thm:GLbranching} to get
\begin{equation}
  \label{eq:86}
  N_{\tau(A)^t,\tau(C')^t,\tau(B)^t,\tau(A')^t,\tau(C)^t,\tau(B')^t} = c^{[\tau(C')^t,\tau(C)^t]_{m}}_{[\tau(A)^t,\tau(A')^t]_{m},
  [\tau(B)^t,\tau(B')^t]_{m}}.
\end{equation}

Since the six fold NL-coefficient are invariant by conjugating the partitions,
\eqref{eq:87} and \eqref{eq:86} implies \eqref{eqn:Nnotzero} as expected.
\qed

\begin{remark}
The earlier version of this work
(\textsf{arXiv:2107.03152v1}) did not use Proposition~\ref{thm:GLbranching}.
We gave a combinatorial proof, perhaps of independent interest, that
connects the celebrated Robinson-Schensted-Knuth algorithm to the ``demotion'' algorithm of \cite{GOY20a}.
\end{remark}

\section*{Acknowledgements}
We thank Winfried Bruns and the {\sf Normaliz} team. 
SG, GO, and AY were partially supported by NSF RTG grant DMS 1937241. 
SG was partially supported by an NSF graduate research fellowship.
AY was supported by a Simons collaboration grant and UIUC's Center for Advanced Study. 

\bibliographystyle{plain}
\bibliography{nl.bib}

\end{document}